\documentclass[12pt]{amsart}

\usepackage{amssymb}

\usepackage[all]{xy}

\makeatletter
\@addtoreset{equation}{section}
\makeatother

\theoremstyle{plain}
\newtheorem{theorem}{Theorem}[section]
\newtheorem{lemma}[theorem]{Lemma}
\newtheorem{proposition}[theorem]{Proposition}
\newtheorem{corollary}[theorem]{Corollary}
\newtheorem{question}[theorem]{Question}
      
\theoremstyle{definition}

\newtheorem{remark}[theorem]{Remark}
\newtheorem{example}[theorem]{Example}

\newcommand{\Cee}{{\mathbb C}}
\newcommand{\Ree}{{\mathbb R}}
\newcommand{\En}{{\mathbb N}}
\newcommand{\kay}{{\Bbbk}}
\newcommand{\Oh}{{\mathbb O}}
\newcommand{\Que}{{\mathbb Q}}
\newcommand{\Tee}{{\mathbb T}}
\newcommand{\Zee}{{\mathbb Z}}

\newcommand{\alp}{\alpha}
\newcommand{\del}{\delta}
\newcommand{\Del}{\Delta}
\newcommand{\eps}{\varepsilon}
\newcommand{\lam}{\lambda}
\newcommand{\vphi}{\varphi}
\newcommand{\sig}{\sigma}

\newcommand{\fA}{\mathcal{A}}
\newcommand{\fB}{\mathcal{B}}
\newcommand{\fC}{\mathcal{C}}
\newcommand{\fH}{\mathcal{H}}
\newcommand{\fL}{\mathcal{L}}
\newcommand{\fN}{\mathcal{N}}

\newcommand{\ck}{\mathrm{K}}
\newcommand{\crk}{\mathrm{RK}}
\newcommand{\cmap}{\mathrm{MAP}}
\newcommand{\csin}{\mathrm{SIN}}
\newcommand{\cp}{\mathrm{P}}
\newcommand{\crp}{\mathrm{RP}}
\newcommand{\crsin}{\mathrm{RSIN}}

\newcommand{\ls}{\mathfrak{s}}

\newcommand{\Fr}{\mathrm{F}}
\newcommand{\GL}{\mathrm{GL}}
\newcommand{\SL}{\mathrm{SL}}
\newcommand{\SO}{\mathrm{SO}}

\newcommand{\Sp}{\mathrm{Sp}}
\newcommand{\Un}{\mathrm{U}}
\newcommand{\mat}{\mathrm{M}}

\newcommand{\spn}{\mathrm{span}}
\newcommand{\ind}{\mathrm{ind}}

\newcommand{\op}{\mathrm{op}}

\newcommand{\alg}{\mathrm{G}}
\newcommand{\cont}{\mathrm{C}}
\newcommand{\cstar}{\mathrm{C}^*}

\newcommand{\fsal}{\mathrm{B}}
\newcommand{\bl}{\mathrm{L}}

\newcommand{\pos}{\mathrm{P}}
\newcommand{\tpos}{\mathrm{T}}
\newcommand{\Tr}{\mathrm{Tr}}
\newcommand{\amTr}{\mathrm{amTr}}
\newcommand{\vn}{\mathrm{VN}}

\newcommand{\til}[1]{\tilde{#1}}
\newcommand{\what}[1]{\widehat{#1}}
\newcommand{\wtil}[1]{\widetilde{#1}}
\newcommand{\wbar}[1]{\overline{#1}}

\begin{document}
 
\author{Brian E. Forrest}
\address{Pure Mathematics, University of Waterloo, Waterloo, ON, N2L 3G1, Canada}
\email{beforrest@uwaterloo.ca}
 
\author{Nico Spronk}
\address{Pure Mathematics, University of Waterloo, Waterloo, ON, N2L 3G1, Canada}
\email{nspronk@uwaterloo.ca}

\author{Matthew Wiersma}
\address{Department of Mathematics and Statistics, University of Winnipeg, 515 Portage Ave, Winnipeg, MB R3B 2E9, Canada}
\email{m.wiersma@uwinnipeg.ca}

\title[Traces on locally compact groups]{Traces on locally compact groups}

\begin{abstract}
We conduct a systematic study of traces on locally compact groups, in particular
traces on their universal and reduced C*-algebras.  We introduce the trace kernel,
and examine its relation to the von Neumann kernel and to small-invariant neighbourhood
($\csin$) quotients.  In doing so, we introduce the class of residually-$\csin$ groups,
which contains both $\csin$ and maximally almost periodic groups.  We examine in
detail the trace kernel for connected groups.  We study traces on reduced C*-algebras,
giving a simple proof for compactly generated groups that existence of such a trace
is equivalent to having an open normal amenable subgroup, and we 
display non-discrete groups admitting unique trace.  

We finish by examining amenable
traces and the factorization property.  We show for property (T) groups that  amenable
trace kernels coincide with von Neumann kernels.  We show for totally disconnected 
groups that amenable trace separation implies the factorization property.  We use amenable
traces to give a simple proof that amenability of the group is equivalent to simultaneous 
nuclearity and possessing a trace of its reduced C*-algebra.
As a final application of the results obtained in the paper, we address the embeddability of group C*-algebras into simple AF algebras.  As a consequence, if a locally compact group is amenable 
and tracially separated (trace kernel is trivial), then its reduced C*-algebra is quasi-diagonal.
\end{abstract}


\thanks{The three authors were partially supported by an NSERC Discovery Grants.}


 \date{\today}
 
 \maketitle
 

\section{Introduction and background}

Tracial states play an indispensable role when studying C*-algebras and are intimately related to many important structural properties. For example, unital quasidiagonal C*-algebras admit amenable traces and every tracial state on a C*-algebra with the weak expectation property (WEP) is amenable. As such, obtaining an understanding of the set of tracial states on a C*-algebra may allow one to deduce further structural properties of the C*-algebra. In this paper, we initiate a systematic study of tracial states on group C*-algebras. Being one of the most prolifically studied classes of C*-algebras, it is unsurprising that this problem has been considered in these special cases before.  If G is a discrete group, then it is well-known that its reduced $\cstar$-algebra $\cstar_r(G)$ admits a tracial state. Moreover, the  uniqueness of  this trace on $\cstar_r(G)$  is linked to the simplicity $\cstar_r(G)$. 

A locally compact group is said to be \textit{C*-simple} if the reduced C*-algebra $\cstar_r(G)$ is simple, or equivalently if every continuous unitary representation $\pi$ of $G$ that is weakly contained in the left regular representation $\lambda$ is weakly equivalent to $\lambda$.  The literature concerning C* simplicity for locally compact groups is very rich. Since a non-trivial connected group is never C*-simple, the focus in studying C*-simple groups has been on discrete groups. Interest in such groups can be traced back to a question of Dixmier who asked in 1967 whether every simple C*-algebra is generated by its projections. Later Kadison conjectured that the reduced C*-algebra of $\mathbb{F}_2$, the free group on two generators, might provide an example of a simple C*-algebra with no non-trivial projections. The simplicity of $\cstar_r(\mathbb{F}_2)$ 
was shown by Powers in \cite{powers}, but he was unable to determine whether or not $\cstar_r(\mathbb{F}_2)$ was without non-trivial projections. This was verified several years later by Pimsner and Voiculescu \cite{PimVoi} completing the proof of Kadison's Conjecture. Following
 the work of Powers, W. L. Paschke and N. Salinas  \cite{paschkes} showed that the reduced C*-algebra of the free product of two groups not
 both of order two is simple and has a unique tracial state. P. de la Harpe, together with M. Bekka and M. Cowling  gave additional example
 of C*-simple discrete  groups with unique traces including non -solvable subgroups of ${\rm PSL}(2,\mathbb{R})$ and non-almost solvable subgroups of  ${\rm PSL}(2,\mathbb{C})$.  See \cite{delaharpe}, \cite{bekkacdlh} and \cite {bekkacdlh1} as well as \cite{delaharpe2}. In an unpublished
 manuscript \cite{Poz}, T. Poznansky   showed that  a linear group $G$ is C*-simple if and only if  $\cstar_r(G)$ admits a unique trace.   More
 recently  Kalantar and Kennedy  \cite{Kalken}  showed that if $\cstar_r(G)$ is simple, then the trace on $\cstar_r(G)$ is always unique. From there a
 complete characterization of when  $\cstar_r(G)$ admits a unique tracial state for a discrete group $G$ was given by Breuillard, Kalantar, Kennedy
 and Ozawa \cite{breuillardkko}.  In particular, they show that a discrete group has the unique trace property if and only if the amenable radical,
 the largest normal amenable subgroup, is trivial. They also show that a discrete group with the uniques trace property that has non-trivial
 bounded cohomology or  non-vanishing $\ell^2$-Betti numbers is C*-simple. Finally, Le Boudec  \cite{leboudec} settled the question of the
 possible equivalence of the unique trace property with C*-simplicity by exhibiting an example of  discrete group $G$ with the unique trace
 property that is not C*-simple.

In contrast with the discrete case, for non-discrete groups  $\cstar_r(G)$ need not always have a tracial state, though it does when $G$ is amenable. For locally compact groups, the problem of when the reduced group C*-algebra admits a tracial state has been studied by various authors. This problem was first solved in the case of separable and connected groups by Ng (see \cite{ng}) and later solved in the compactly generated case by the present authors (see \cite{forrestsw}) where we showed that when $G$ is compactly generated, $\cstar_r(G)$  admits a tracial state if and only if $G$ has an open amenable subgroup. We also speculated that the assumption of $G$ being compactly generated was not necessary. While  \cite{forrestsw} was being considered for publication, Kennedy and Raum \cite{kennedyr}, using entirely different techniques were able,  to remove the assumption that $G$ was almost compactly generated to show that our speculation was correct. While the result of Kennedy and Raum is more general, our approach had the advantage in that it gave more detailed information about the structure of groups for which $\cstar_r(G)$ possesses a trace. Moreover, it gave us tools to address the more general question concerning  the nature of the traces on the full group C*-algebra $\cstar(G)$. 

Our approach to systematically studying tracial states on the full group C*-algebra of a locally compact group relies on the introduction of the ``trace kernel'', which is the closed and normal subgroup $N_{\Tr}$ of the locally compact group $G$ corresponding to group elements that are not separated from the identity by tracial states. The introduction of the trace kernel allows us to relate the set of tracial states on the full group C*-algebra with the structure of the corresponding locally compact group.  In general, the smaller that $N_{\Tr}(G)$ is the richer the collection of traces. Moreover, we show amongst other things that $G/N_{\Tr}(G)$ is almost SIN. In particular, if $G$ is connected, then $G/N_{\Tr}(G)= V \times K$ where $V$ is a vector group and $K$ is compact.

The question of when a C*-algebra can be embedded into an AF-algebra has a long history. It can be traced back to Elloitt's original classification of AF-algebras. Early work in this direction includes Spielberg's proof that a separable, residually finite, type I C*-algebra can be embedded into an AF algebra  \cite{spiel}. Later Pimsner and Voiculescu \cite{PimVoi2} showed that the irrational rotation algebra is AF-embeddable, prompting   Effros \cite{Effros} to ask about a possible abstract classification of such embeddable algebras. 

Despite the considerable effort of many people, the question of  embeddability into a generic AF-algebra is difficult with as of yet no satisfactory general theory to rely on. At this point the working conjecture is that a C*-algebra is AF-embeddable if and only it is separable, exact, and quasidiagonal. In contrast, if we restrict our attention  to embeddability into a unital, simple AF algerba, we can take advantage of recent deep work of C. Schafhauser  \cite{schafhauser}. Amongst other things, Schafhauser showed that  if $A$ is a separable, exact C*-algebra which satisfies the Universal Coefficient Theorem (UCT) and has a faithful, amenable trace, then $A$ admits a trace-preserving embedding into a simple, unital AF-algebra with unique trace. In particular, for any countable, discrete, amenable group $G$, the reduced group C*-algebra $\cstar_r(G)$ admits a trace-preserving embedding into the universal UHF-algebra. For $G$ non-discrete,  I. Belti\c{t}\u{a} and D. Belti\c{t}\u{a} \cite{beltitab} show that if $G$ is connected, second countable and solvable, that embeddability into an AF-algebra is possible if and only if $G$ is abelian.

\subsection*{Notation and summary of results}

For a locally compact group $G$, we let
\begin{align*}
\pos(G)&=\{u:G\to\Cee\;|\;u\text{ is continuous and positive definite}\} \\
\pos_1(G)&=\{u\in\pos(G):u(e)=1\} \\
\tpos(G)&=\{u\in\pos_1(G):u(st)=u(ts)\text{ for each }s,t\text{ in }G\}
\end{align*}
denote, respectively, the {\it positive definite functions},
{\it states}, and the {\it traces}, each on $G$.  By a well-known correspondence
these sets may be identified with the positive functionals on the enveloping C*-algebra
$\cstar(G)$, respectively, the states, and the tracial states.  
See, for example \cite[\S 7.1]{pedersen}.

In Section \ref{sec:traker}, we introduce the trace kernel $N_\Tr$.  We place this
in context against two other kernels, the small-invariant neighbourhood kernel,
$N_\csin$; and the maximally almost periodic (MAP), or von Neumann, kernel, $N_\cmap$.
We introduce the new class of residually small invariant neighbourhood (SIN) groups
$[\crsin]$.  We show that this class properly contains the union of classes 
$[\cmap]\cup[\csin]$, and that $N_\Tr\subseteq N_\csin\subseteq N_\cmap$, with
$N_\Tr= N_\csin$ when $G$ is compactly generated.  We also show that
$G/N_\Tr$ is almost-SIN, and hence quasi-SIN.

In Section \ref{sec:congro}, we apply results from the previous section to learn about the
structure of $N_\Tr$ and of $G/N_\Tr$ for connected groups, with a primary focus on connected Lie groups.

Sections \ref{sec:traker} and \ref{sec:congro} contain many examples that illustrate
the complexities and limitations of the results.

In Section \ref{sec:redtra}, we consider reduced traces, those on the reduced group 
C*-algebra.  Using some observations from Section \ref{sec:traker}, we simplify  the proof
from our preprint \cite{forrestsw}
that for compactly generated groups, the existence of a reduced trace is equivalent
to $G$ admitting an open normal amenable subgroup.  As previously mentioned, this fact was simultaneously shown
in \cite{kennedyr} without the compact generation assumption.  Our methods complement theirs and both can be used in tandem to learn about the structures of reduced traces
in many examples.  For example, we can produce examples of non-discrete 
groups admitting a unique reduced trace.

In Section \ref{sec:ametra} we consider amenable traces.  We give a (mostly)
self-contained function-theoretic approach to their definition.  We introduce
the amenable trace kernel $N_{\mathrm{amTr}}$.  If $G$ has property (T),
then $N_{\mathrm{amTr}}$ coincides with the von Neumann kernel $N_\cmap$, hence
$G$ is amenably tracially separated exactly when it is maximally almost periodic.  We then proceed
to examine the relationship of amenable trace separation with the factorization property
of Kirchberg \cite{kirchberg}.  In particular, we learn that for totally disconnected groups
amenable trace separation implies the factorization property.  Finally, we
use amenable traces on the reduced C*-algebra to gain a simple proof of the 
characterization,
first due to Ng \cite{ng}, that $G$ is amenable if and only if $\cstar_r(G)$ is nuclear and admits a trace.

In Section 6, we turn our attention to studying structural properties of group C*-algebras, in particular embeddability into simple, unital AF algebras. Appealing to deep results arising from classification theory of C*-algebras, we show the reduced group C*-algebra of a second countable locally compact group embeds inside of a simple, unital AF algebra if and only if the group is amenable and the tracial kernel $N_{\Tr}$ is trivial. As a consequence, we find the reduced C*-algebras of tracially separated, amenable groups are quasidiagonal. We additionally demonstrate that the full C*-algebra of a non-compact property (T) group cannot be embedded inside of a unital, simple AF algebra.

\subsection{First observations.}
We note two obvious cases:
\begin{align}\label{eq:abelcom}
&G\text{ abelian: } && \tpos(G)=\pos_1(G)\cong \mathrm{Prob}(\what{G}) \\
&G\text{ compact: } && 
\tpos(G)=\wbar{\mathrm{conv}\left\{\frac{1}{d_\pi}\chi_\pi:\pi\in\what{G}\right\}} \notag
\end{align}
where $\what{G}$ denotes the dual object, a locally compact group in the abelian case,
and in the compact case a full set of irreducible representations, each
of necessarily finite dimension $d_\pi$ and with character $\chi_\pi=\Tr\circ\pi$.

If $N$ is a closed normal subgroup of $G$, we let $q_N:G\to G/N$ denote the quotient map.

Given $u\in\pos_1(G)$ we let $(\pi_u,\fH_u,\xi)$ denote the 
Gelfand-Naimark-Segal (GNS) representation
associated with $u$, so $u=\langle\pi_u(s)\xi|\xi\rangle$. 
We record an observation that is well-known to specialists.

\begin{lemma}\label{lem:kernel}
\begin{itemize}
\item[(i)] If $u\in \pos_1(G)$ then $N_u=u^{-1}(\{1\})$ is a closed subgroup of $G$,
$u$ is constant on double cosets $N_utN_u$, and $\ker\pi_u=\bigcap_{t\in G}tN_ut^{-1}$.
\item[(ii)] If $u\in\tpos(G)$ then $N_u=\ker\pi_u$ is normal, and 
$u\in\tpos(G/N_u)\circ q_{N_u}$.
\end{itemize}
\end{lemma}

\begin{proof}
{\bf (i)} By uniform convexity
$N_u=\{s\in G;\pi_u(s)\xi=\xi\}$, which is easily shown to be a subgroup.  Furthermore
$\pi$ is constant on double cosets $N_utN_u$.  If $s\in \bigcap_{t\in G}tN_ut^{-1}$ 
and $t\in G$ then $\pi_u(s)\pi_u(t)\xi=\pi_u(t)\pi_u(t^{-1}st)\xi=\pi_u(t)\xi$,
and cyclicity of $\xi$ provides that $\pi_u(s)=I$, so $N_u\subseteq\ker\pi_u$.
The converse inclusion is obvious.

{\bf (ii)} The trace condition provides normality of $N_u$.  Hence
$N_u=\ker\pi$.  Furthermore, $\pi$
is constant on cosets of $N_u$ and hence defines a representation $\tilde{\pi}$
of $G/N_u$.  Then $u=\langle\tilde{\pi}_u\circ q_{N_u}(\cdot)\xi,\xi\rangle$. \end{proof}

\section{The tracial kernel}\label{sec:traker}
Let $G$ be a locally compact group.
We define the {\it tracial kernel} of $G$ by
\[
N_\Tr=N_\Tr(G)=\bigcap_{u\in\tpos(G)}N_u.
\]
We say that $G$ is {\it tracially separated} if $N_\Tr=\{e\}$.
It is evident that
\begin{equation}\label{eq:tracered}
\tpos(G)=\tpos(G/N_\Tr)\circ q_{N_\Tr}
\end{equation}
and that  $G/N_\Tr$ is the maximal tracially separated quotient of $G$.

\subsection{$\cp$-kernels.}
We first wish to position the trace kernel in relation to other classes of kernels.
Let $[\cp]$ be a class of locally compact groups.  For example we let
$[\cmap]$ denote the class of {\it maximally almost periodic} groups and
$[\csin]$ that of {\it small invariant neighbourhood} groups.  We let
$\fN_\cp(G)$ denote the set of ``co-$\cp$" closed normal subgroups $N$ in $G$,
i.e.\ those for which $G/N\in[\cp]$, and define the $\cp$-kernel by
\[
N_\cp=N_\cp(G)=\bigcap_{N\in\fN_\cp(G)}N.
\]
The most well-known example is the {\it von Neumann kernel}, $N_{\cmap}$.
Furthermore, $G$ is called {\it minimally almost periodic} if $N_\cmap=G$ or, equivalently, if the trivial representation is the unique irreducible finite dimensional representation of $G$.

We do not necessarily have that $G/N_\cp\in[\cp]$, but $G/N$ is {\it residually-}$\cp$, i.e.\
the intersection of co-$\cp$ subgroups.  In fact, if $N$ is a residually-$\cp$ closed normal
subgroup then
\[
N=\bigcap_{M\in\fN_\cp(G/N)}(q_M\circ q_N)^{-1}(M)\subseteq N_\cp(G)
\]
so {\it $N_\cp$ is the smallest closed normal subgroup for which $G/N_\cp$ is
residually-$\cp$}.  If we let $[\crp]$ denote the class of residually-$\cp$ groups,
then $N_{\crp}=N_\cp$.  In particular, we have an apparently novel class $[\crsin]$.

\begin{theorem}
We have inclusion of classes $[\cmap]\subseteq[\crsin]$.
\end{theorem}

\begin{proof}
Let $G\in[\cmap]$.  We shall show for each finite dimensional representation
$\pi:G\to\Un(d)$, that $G_\pi=G/\ker\pi\in[\csin]$.  As the intersection of such kernels
separates points of $G$, this will show that $G\in[\crsin]$.

The connected component $G_0\in[\cmap]$, and
the Fruedenthal-Weil Theorem (see \cite[12.4.8]{palmer} or \cite[16.4.4]{dixmier}) 
shows that $G_0\cong V\times K$
for a vector group $V$ and compact $K$.  Then $K$ is characteristic in $G_0$, hence
normal in $G$, while $V$ can arranged to be normal by 
Robertson and Wilcox \cite[Theo.\ 2]{robertsonw}.
Furthermore, \cite[Theo.\ 1]{robertsonw} shows that the centralizer $C=C_G(V)$
is of finite index, hence open, in $G$.  Notice that $G_0\subseteq C$.  
We first establish that the open image of $C$ in $G_\pi=G/\ker\pi$,
which is isomorphic to $C_\pi=C/\ker\pi|_C$, is in $[\csin]$.

We let $\fB$ be a neighbourhood base at the identity for $C/G_0$ consisting of 
compact open subgroups.  Each $H_B=q_{G_0}^{-1}(B)\subseteq C$ is almost connected, 
hence by \cite[(2.9)]{grosserm2} of the form $V\times K_B$,
which is a direct product as $V$ is central in $C$.  
Notice that each $K_B$ is open in $C$.
We have that $\bigcap_{B\in\fB}K_B
=K$, and each $\pi(K_B)$ is a compact Lie group.  For $B$ in $\fB$ we have
that $\pi(K_B)/\pi(K)$ is a Lie image of totally disconnected compact group, hence finite.
It follows that $\pi(K_B)=\pi(K)$ for $B$ in $\fB$ sufficiently small, and we fix such a $B$.
Then $\pi(K_B)=\pi(K)$, and hence $q_\pi(K)=q_\pi(K_B)$ is open in $C_\pi$. 
Now $q_\pi(K)$, being homeomorphic to the normal subgroup $\pi(K)$ in the compact
group $\wbar{\pi(C)}$ in $\Un(d)$, admits a base $\fC$ of $C_\pi$-invariant neighbourhoods
of the identity, while $q_\pi(V)$ is central in $C_\pi$. Consider the 
open subgroup of $C_\pi$ given by
\[
q_\pi(G_0)=q_\pi(V\times K)
=q_\pi(V)q_\pi(K). 
\]
Any neighbourhood $U$ of the identity in $q_\pi(G_0)$ admits $W$ in $\fC$
such that $\pi(W)\subseteq U$, and $(U\cap\pi(V))\cap\pi(W)$ is a
$C_\pi$-invariant neighbourhood of the identity.  In other words, $C_\pi\in[\csin]$.

Now $C_\pi$ is open and of finite index in $G_\pi$.  If $U$ is a relatively compact
conjugation-invariant neighbourhood of the identity in $C_\pi$, then
$\bigcap_{tC_\pi\in G_\pi/C_\pi}tUt^{-1}$ is a well-defined and
conjugation-invariant neighbourhood of the identity for $G_\pi$.  
It follows that $G_\pi\in[\csin]$.
\end{proof}

The following may be known, but may easily seen
from our proof.

\begin{corollary}\label{cor:tdmap}
If $G\in[\cmap]$ and totally disconnected, then $G$ is residually discrete.
If, additionally, $G$ is compactly generated
then $G$ is residually finite.
\end{corollary}

\begin{proof}
Let $\pi:G\to \Un(d)$ be a representation.  The proof above 
shows that $\ker\pi$ contains some open subgroup $B$.
Hence if $G$ is compactly generated, then
$\pi(G)$ is finitely generated, and \cite{malcev,alperin} shows that
$\pi(G)$ is residually finite.
\end{proof}

We let $[\ck]$ denote the class of compact groups and we obtain inclusions:
\begin{equation}\label{eq:rsinworld}
\xymatrix{
[\ck] \ar[r]\ar[dr] & [\crk] \ar[r] & [\cmap]=[\mathrm{RMAP}] \ar[d] \\
& [\csin] \ar[r] & [\crsin].
}
\end{equation}
We get accordingly, the following containment of kernels:
\begin{equation}\label{eq:kernels}
N_{\csin}\subseteq N_{\cmap}\subseteq N_{\ck}.
\end{equation}
Proper inclusions for the top row in (\ref{eq:rsinworld})
are shown, respectively, by 
$\Ree$, and the divisible discrete group $\Que$.
The following construction can be fine-tuned to create 
examples which show that no inclusion of classes is omitted in (\ref{eq:rsinworld}).

\begin{example}\label{ex:gammaF}
We consider a sequence of pairs $(\Gamma_n,F_n)$ where each $\Gamma_n$ is discrete,
and each $F_n$ is a finite group of automorphsims which contains one element
$\alp_n$ for which the commutator $\{s\alp_n(s^{-1}):s\in\Gamma_n\}$ is infinite.
Hence in each semi-direct product $\Gamma_n\rtimes F_n$, there is an element
with infinite conjugacy class.  Let
\[
\Gamma=\bigoplus_{n=1}^\infty\Gamma_n,\quad
F=\prod_{n=1}^\infty F_n\quad\text{and}\quad G=\Gamma\rtimes F
\]
where $\Gamma$ is restricted direct product (catagorical direct sum), and the product
group $F$ is compact and acts on $\Gamma$ coordinate-wise.  It is then easy to see that
\begin{itemize}
\item each neighbourhood of the identity on $G$ contains a neighbourhood
$\{e\}\rtimes\prod_{n=m}^\infty F_n$ and hence
an element with infinite conjugacy class in a discrete subgroup, so $G\not\in[\csin]$; and
\item each normal subgroup $N_m=\left(\bigoplus_{n=m}^\infty\Gamma_n\right)\rtimes
\left(\prod_{n=m}^\infty F_n\right)$ has discrete quotient, so $G\in[\crsin]$.
\end{itemize}

{\bf (i)} The matrix $\begin{bmatrix} 0 & -1 \\ 1 & 0\end{bmatrix}$ is 
of order $4$ but squares to a central element, and is easily checked
the have an infinite conjugacy class in $\SL_2(\Zee)$, thus in $\SL_2(\Que)$, 
and hence its inner automorphism $\alp$ has infinite commutator.  
We form $G$, as above, with choices $\Gamma_n=\SL_2(\Que)$ and 
$F_n=\langle\alp_n\rangle$, where $\alp_n=\alp$.  
Since $G$ contains (closed) copies of $\SL_2(\Que)$, which is minimally almost periodic
(as first shown by von Neumann and Wigner in  \cite{vonneumannw}),  $G\not\in[\cmap]$.  

{\bf (ii)} If in contrast to above, we let each $\Gamma_n=\SL_2(\Zee)$ instead of $\SL_2(\Que)$ , then
$\alp$-invariant congruence subgroups of $\SL_2(\Zee)$ show that
each $\SL_2(\Zee)\rtimes\langle\alp\rangle$ is residually finite.  We obtain that
$G\in[\crk]$.  

{\bf (iii)}  The classical example of Murakami
\cite{murakami} (also \cite[12.6.4]{palmer}), with
each $\Gamma_n=\Zee$ and $F_n=\{1,\alp\}$ 
with $\alp(n)=-n$, is also in $[\crk]\setminus [\csin]$ and is, furthermore, amenable.
If here, we instead let each $\Gamma_n=\Que$ we get a group in
$[\cmap]\setminus([\crk]\cup[\csin])$.
\end{example}

We now position our trace kernel with respect to the kernels of (\ref{eq:kernels}),
and gain information about the size of $\tpos(G)$ in the process.  We shall
say that $\tpos(G)$ is {\it infinite dimensional}, if it contains an infinite linearly independent
set.

\begin{proposition}\label{prop:trasin}
\begin{itemize}
\item[(i)] We have that $N_\Tr\subseteq N_{\csin}$.  
\item[(ii)] If $N_\csin$  is non-open, then $\tpos(G)$ is infinite dimensional.
\item[(ii$'$)] If $N_\cmap$ is of infinite index in $G$, then
$\tpos(G)$ is infinite dimensional.
\end{itemize}
\end{proposition}

\begin{proof}
{\bf (i)} Let $q:G\to H$ be a quotient  homomorphism where $H\in[\csin]$. 
Let $\lam:H\to\Un(\bl^2(H))$ denote the left regular
representation.  Given any $r\in H\setminus \ker\vphi$ there is a conjugation-invariant
symmetric neighbourhood $U$ of $e$ in $H$ for which $q(r)\not\in U^2$.  Then
for $s$ in $H$
\[
u_U(s)=\frac{1}{m(U)}\langle \lam(s)1_U|1_U\rangle=\frac{m(sU\cap U)}{m(U)}
\]
defines a trace, since $H$ is unimodular and
$t^{-1}(tsU\cap U)t=sUt\cap t^{-1}Ut=stU\cap U$ for $s,t$ in $H$.  Also
$u_U\circ q(r)=0\not=1$.  Hence $\ker q\supseteq N_\Tr$.

{\bf (ii)} First, suppose there is $N$ in $\fN_\csin$ be non-open, so $G/N$ is a non-discrete
$\csin$-group.  Then there is a sequence $U_1,U_2,\dots$ of
relatively compact symmetric open neighbourhoods of the identity
in $G/N$ for which $\wbar{U_{k+1}^2}\subsetneq U_k$.  Then
the set of elements $u_{U_n}\circ q_N$, as above, is linearly independent.

If each element $N$ in $\fN_\csin$ is open, but $N_\csin$ is not open, then
we may find a strictly decreasing sequence $N_1,N_2,\dots$
of elements of open normal subgroups.  Then $1_{N_1},1_{N_2},\dots$ is 
 a linearly independent sequence in $\tpos(G)$.

{\bf (ii$'$)} We have that $Q=G/N_\cmap$ is infinite and maximally almost periodic,
since it is residually maximally almost periodic.
Then the almost periodic compactification $Q^{ap}$ is infinite.
Let $\iota:Q\to Q^{ap}$ be the compactification map.
With $U_1,U_2,\dots$ chosen in $Q^{ap}$ as above, the set of elements 
$u_{U_n}\circ\iota\circ q_{N_\cmap}$
is linearly independent.
\end{proof}

\subsection{The role of compact generation.}
The next observation is of fundamental importance.  The main aspect
of the argument is very similar to one partially attributed by Hofmann and Mostert
\cite[Prop.\ 12.2]{hofmannm} to Freudenthal  (see
\cite[12.4.16]{palmer}, for example), where it is shown for compactly generated
$G$ that $G\in[\cmap]$ implies that $G\in[\csin]$.

\begin{proposition}\label{prop:comgen}
Suppose that $G$  is compactly generated, and
admits a continuous homomorphism $\vphi:G\to H$, where $H$
is a topological group admitting a family of conjugation invariant open
sets $\fB$ with $\bigcap_{B\in\fB}\wbar{B}=\{e\}$.
Then $G\in[\csin]$.

In particular if $G$ is compactly generated we have
\begin{itemize}
\item[{\bf (i)}] for each $u$ in $\tpos(G)$ that $N_u\in\fN_\csin$; and
\item[{\bf (ii)}] if $G\in[\crsin]$, then $G\in[\csin]$.
\end{itemize}
\end{proposition}

\begin{proof}
The assumption of compact generation provides a compact symmetric
neighbourhood $K$ of $e$ in $G$ for which $G=\bigcup_{n=1}^\infty K^n$.
Let $W$ be an open neighbourhood of $e$ with $W\subseteq K^3$.
We have $K^3\setminus W\subseteq
\bigcup_{B\in\fB}(G\setminus \vphi^{-1}(\wbar{B}))$, and hence 
$K^3\setminus W\subseteq G\setminus \vphi^{-1}(\wbar{B})\subseteq
G\setminus \vphi^{-1}(B)$ for some $B$.  Let
$U=K^3\cap \vphi^{-1}(B)\subseteq W$.
If $x\in K$ we have on one hand that $xUx^{-1}\subseteq\vphi^{-1}(B)$, and
on the other that $xUx^{-1}\subseteq K^3$, and hence
$xUx^{-1}\subseteq U$.  As $W$ is arbitrarily small, and $K$ generates
$G$, this construction provides
a conjugation invariant base of neighbourhoods of $e$, i.e.\
 $G\in[\csin]$.

{\bf (i)} We let $(\pi_u,\fH_u,\xi)$ be the GNS representation associated with $u$. Consider the group of unitaries $H=\pi_u(G)$ which  is a topological
group in the relativized weak operator topology (equivalently, strong operator topology).  
The trace condition shows that for $\eps>0$ that each $B_\eps=
\{x\in H:|\langle x\xi|\xi\rangle-1|<\eps\}$, so $\wbar{B_\eps}\subseteq
 \{x\in H:|\langle x\xi|\xi\rangle-1|\leq\eps\}$.
As in Lemma \ref{lem:kernel}
we see that $\bigcap_{\eps>0}\wbar{B_\eps}=\{I\}$.  Let $\vphi=\til{\pi}_u:G/N_u\to H$.

{\bf (ii)} Here we let $H=\prod_{N\in\fN_{\csin}(G)}G/N$ be the product topological
group; $\fB$ be the family of open neighbourhoods of $e$ which are proper 
conjugation-invariant
sets on a finite set of coordinates, and the whole group on the remaining ones; and
$\vphi:G\to H$ be given by $\vphi(s)=(q_N(s))_{N\in\fN_{\csin}(G)}$.
\end{proof}

\begin{remark}
Conclusion (i) can fail if $G$ is not compactly generated.
Consider $G=(\bigoplus_{n=1}^\infty\Gamma_n)\rtimes\prod_{n=1}^\infty F_n$,
as in Example \ref{ex:gammaF}.  Each quotient $\Gamma_n\rtimes F_n$
admits trace $1_{\{(e,e)\}}$, so $u((x_n,\sigma_n)_{n=1}^\infty)
=\sum_{n=1}^\infty\frac{1}{2^n}1_{\{(e,e)\}}(x_n,\sigma_n)$ defines a trace
with $N_u=\{e\}$.
\end{remark}

Proposition \ref{prop:comgen} shows that
for compactly generated groups we get containment diagram
\begin{equation}\label{eq:cgclass}
[\crk]\longrightarrow
[\cmap]\longrightarrow [\csin]=[\crsin]
\end{equation}
If $G$ is either connected, or totally disconnected, then $G\in[\cmap]$ implies
that $G\in[\crk]$.  Indeed this follows from the Freudenthal-Weil Theorem and 
Corollary \ref{cor:tdmap}, respectively.  

\begin{example} 
The first two inclusions of (\ref{eq:cgclass}) remain proper
amongst compactly generated groups.

{\bf (i)} Let $C_3=\Zee/3\Zee$ act on $\Ree^2$, 
by the obvious rotations.  Any normal subgroup of $G=\Ree^2\rtimes C_3$
must intersect $\Ree^2$ as a $C_3$-invariant subgroup.  Any non-trivial
such subgroup must be dense in $\Ree^2$.  It follows that $N_\ck=\Ree^2$,
though this group is maximally almost periodic.

{\bf (ii)} Certain Baumslag-Solitar groups, such as  $G=BS(2,3)$, seem to be amongst the 
most famous examples of finitely generated, non-residually finite, or equivalently 
non-maximally almost periodic, groups  (e.g.\ Corollary \ref{cor:tdmap}).
In fact, in \cite{baumslags} it is indicated that for $G$ above
that the solvable group $G/G''$
also satisfies a condition which entails that it is not residually finite.

{\bf (iii)} In \cite[2.7]{decornulier} de Cornulier presents
a construction of a finitely generated
linear group with  a quotient by a certain proper subgroup of its centre that  is not 
residually finite.  This illustrates some limitations of \cite{leptinr} where it is shown that
a $\cmap$-group modulo its centre is in $[\cmap]$, and also of
\cite{malcev,alperin} that finitely generated linear groups are residually finite.
\end{example} 

\begin{corollary}\label{cor:comgen}
If $G$ is compactly generated, then $N_\Tr=N_\csin$.
\end{corollary}

\begin{proof}
Proposition \ref{prop:comgen} (i) shows that $N_\csin\subseteq N_\Tr$.
Combine with Proposition \ref{prop:trasin} (i).
\end{proof}

As of yet, we do not have an example of a tracially separated non-$\crsin$
locally compact group $G$, hence no example for which containment
$N_\Tr\subseteq N_\csin$ is proper.  Example \ref{ex:finiteindex}, below
will show for an infinite group that we can have $N_\Tr$ proper and of finite index,
and hence $\tpos(G)$ is finite dimensional.  
It would be interesting to know if there are any discrete $G$ for which
$\tpos(G)$ is non-trivial, but finite-dimensional.  A plausible candidate
is $G=\SL_2(\Que)$.

We say that $G$ is an {\em almost-$\csin$} group if there is a net 
$(U_\alp)_\alp$ of relatively compact
open sets, satisfying that $U_\alp\searrow\{e\}$ (i.e.\
is eventually contained in any neighbourhood of $e$), and
satisfies the asymptotic invariance property
\[
\lim_\alp\sup_{t\in K}\frac{m(tU_\alp t^{-1}\triangle U_\alp)}{m(U_\alp)}=0
\]
for each compact $K\subseteq G/N_\Tr$, where $m$ denotes Haar measure.
A mildly weaker form of this condition is noted by Stokke in \cite[\S 3]{stokke}, where 
sets $U_\alp$ are assumed to be of finite measure, but can be replaced by relatively
compact open sets thanks to regularity; also see Thom \cite[Rem. 3.1]{thom}.   In both of
the above references convergence is assumed to be pointwise, rather than
uniform on compact sets.

\begin{proposition}\label{prop:asin}
The group $G/N_\Tr$ is  almost-$\csin$.
\end{proposition}

\begin{proof}
Every compactly generated open subgroup $H$ of $G/N_\Tr$ 
satisfies $\tpos(H)\subseteq\tpos(G/N_\Tr)|_H$, so $H\in[\csin]$, and hence $H$ is 
unimodular.

Let $W$ be a relatively compact open neighbourhood of $eN_\Tr$ and 
let $K\subseteq G/N_\Tr$ be compact.
Then $H=\langle WK\rangle\in[\csin]$.  Hence, there is an $H$-conjugation-invariant 
neighbourhood $U=U_{K,W}$ of $eN_\Tr$ contained in $W$.  It is now clear that we can build the desired net accordingly.
\end{proof}

Notice that the net $f_\alp=\frac{1}{m(U_\alp)}1_{U_\alp}$ shows that $G/N_\Tr$
enjoys the quasi-$\csin$ property of Losert and Rindler \cite{losertr}.  
It is noted in \cite{stokke} that the 
almost-$\csin$ condition implies unimodularity, and is equivalent
to a strong quasi-$\csin$ property:
$\bl^1(G)$ admits a bounded approximate identity consisting
of normalized indicator functions.

Summarizing results of
this section, we get the following implications for $G$:
\begin{align*}
[\cmap]\;&\Rightarrow \;[\crsin] \;\Rightarrow \;\text{tracial separation} \\
&\Rightarrow \;\text{unimodular \& almost-}\csin\;\Rightarrow\;
\text{quasi-}\csin.
\end{align*}
Non-abelian solvable connected groups, being amenable, are quasi-$\csin$,
but frequently non-unimodular, and always lack trace separation by Remark
\ref{rem:connected}, below.  We do not know if either of the middle two implications
are equivalences.

\section{Connected  groups}\label{sec:congro} 

Now let us suppose that $G$ is connected,
hence compactly generated.  Then Corollary \ref{cor:comgen} and
 Freudenthal-Weil Theorem \cite[12.4.8]{palmer} 
 tell us that
\begin{equation}\label{eq:trakercon}
N_\Tr=N_\csin=N_\cmap\text{ and }G/N_\Tr=V\times K
\end{equation}
where $V$ is a vector group and $K$ is compact.  Hence
$G$ admits a unique trace, i.e.\ $\tpos(G)=\{1\}$, if and only if it is 
{\it minimally almost periodic}, i.e.\ 
$N_\cmap=G$.  Notice that (\ref{eq:tracered}) and (\ref{eq:abelcom}) tell us that
\begin{equation}\label{eq:tracon}
\tpos(G)\cong\wbar{\mathrm{conv}}\left[\pos_1(V)\otimes 
\left\{\frac{1}{d_\pi}\chi_\pi:\pi\in\widehat{K}\right\}\right].
\end{equation}
Indeed, we consider the GNS triple $(\pi_u,\xi,\fH_u)$ of
$u$ in $\tpos(V\times K)$.  Factor $\pi_u(v,k)=\pi_u(v,e)\pi_u(0,k)$,
and then $\langle \pi_u(0,\cdot)\xi|\xi\rangle$ is in the norm-closed
convex hull of normalized characters on $K$.  The associated reducing projections
each commute with $\pi_u(v,e)$ for each $v$ in $V$.

We aim to get more precise descriptions of $V$ and $K$.

If $G$ is a connected Lie group, it admits a maximal solvable normal connected
subgroup $R$, which is always closed.  There is a semisimple connected subgroup
$S$, called a Levi complement, satisfies $G=RS$ and that $R\cap S$ is central
and discrete.  Though $S$ need not be closed, it admits a Lie group structure
by which the inclusion $S\hookrightarrow G$ is continuous.
Two Levi complements are conjugate by an element from the reductive radical 
$[R,G]$, which is connected and normal.  A Levi complement
$S$ contains a maximal connected normal compact subgroup $S_c$ which is the integral
subgroup of the compact ideal $\ls_c$ in its semisimple Lie algebra $\ls$.  Let
$S_{nc}$ denote the integral subgroup of the complementary ideal
of $\ls_c$ in $\ls$.  Then $S_{nc}$ commutes with $S_c$, the intersection
$S_{nc}\cap S_c$ is finite and central in $S$, and $S=S_cS_{nc}$.  Using (\ref{eq:trakercon}), and
in mildly differing notation, it is proved by Shtern \cite{shtern} that
\begin{equation}\label{eq:shtern}
N_\Tr=\wbar{[R,G]S_{nc}}.
\end{equation}
Notice that if $G$ is solvable, then $R=G$, $N_\Tr$ is the closure of the derived
subgroup, and $G/N_\Tr$ is abelian.

\begin{proposition}\label{prop:connlie}
Let $G$ be a connected Lie group.  Then in (\ref{eq:trakercon}) we have that
 $V$ is the maximal vector quotient of $G$, which is also a (not necessarily
maximal) vector subgroup of $R/\wbar{[R,G]}$, and $K$ is the maximal compact 
quotient of $G$ whose semisimple part is a quotient of $S_c$ by a finite group.

If, moreover, $G$ is simply connected, then $V\cong R/[R,G]$ and $K\cong S_c$.
\end{proposition}

\begin{proof}
For $r,r'$ in $R$ and $s,s'$ in $G$ we have
\[
rs[R,G]=sr[R,G]\text{ and hence }rsr's'[R,G]=rr'ss'[R,G].
\]
Thus the map
\[
(r\wbar{[R,G]},sS_{nc})\mapsto rs\wbar{[R,G]S_{nc}}
:(R/\wbar{[R,G]})\times (S/S_{nc})\to G/N_\Tr
\]
is a continuous homomorphism with kernel
\[
D=\{(s\wbar{[R,G]},s^{-1}S_{nc}):s\in R\cap S\}.
\]
We let $L=S/S_{nc}=S_cS_{nc}/S_{nc}=S_c/(S_c\cap S_{nc})$.
The abelian group $R/\wbar{[R,G]}$ admits trivial action by $S$ and
decomposes as a direct product $WT$, where
$W$ is a vector group and $T$ is compact.  Thus $G/N_\Tr$ is the continuous
isomorphic image of
\[
(W\times T\times L)/D=(V\times U\times T\times L)/D
\cong V\times K
\]
where $V$ is a complementary subspace to  $U=\spn_\Ree(W\cap \wbar{S_{nc}})$ in $W$
and $K=(U\times T\times L)/D$.

If $G$ is simply connected, then so too are $R$ and $S$, and the latter
is a direct product $S_c S_{nc}$.  Furthermore, $[R,G]$ is a normal integral
subgroup hence itself closed, with $R/[R,G]$ a vector group
(see \cite[\S 11.2,\S 14.5]{hilgertn}).  Thus $K=S_c$.
\end{proof}

\begin{remark}
In \cite{shtern}, a further remarkable fact is shown.  If we consider the discretization
$G_d$ of a connected Lie group, then
\[
N_\cmap(G_d)=[R,G]S_{nc}.
\]
Notice that Proposition \ref{prop:trasin}, above, shows that $N_\Tr(G_d)=
N_\csin(G_d)=\{e\}$.
We shall indicate in our examples, below, some situations in which this
differs from $N_\Tr(G)$.
\end{remark}

We  now address minimal almost periodicity of $G$.  

\begin{corollary}
A connected Lie group $G$ is minimally almost periodic if and only if
$S_c=\{e\}$ and the image of $S$ in $G/\wbar{[R,G]}$ is dense.
In particular, this holds if $S_c=\{e\}$
and $\wbar{[R,G]}=R$.
\end{corollary}

\begin{proof}
If $S_c\not=\{e\}$, then $G/R$ admits a non-trivial quotient of $S_c$,
as a quotient.
The proof of Proposition \ref{prop:connlie} shows that
$Q=G/\wbar{[R,G]}$ is a continuous isometric image of
\[
(R/\wbar{[R,G]}\times S)/D.
\]
If the image of $S$ is not dense in $Q$, then $Q$ has an non-trivial
abelian quotient.  These two situations are the only pair
of obstructions to minimal almost periodicity.
\end{proof}

\begin{example}\label{ex:Lie}
We consider examples to illustrate the complications of situation above.  

We let 
$\wtil{\SL}_2(\Ree)$ be the simply connected covering group of 
$\SL_2(\Ree)$.  Let $Z$ denote the centre of $\wtil{\SL}_2(\Ree)$ which
admits an isomorphism $j:\Zee\to Z$, and satisfies that 
$\wtil{\SL}_2(\Ree)/Z\cong \SL_2(\Ree)$.  

{\bf (i)}  Let $G=[\Ree\times\wtil{\SL}_2(\Ree)]/D$ where $D=\{(n,j(n)):n\in\Zee\}$.
Then $R=[\Ree\times\{I\}]/D\cong\Ree$ with $[R,G]=\{e\}$ and $S=S_{nc}=
[\{0\}\times\wtil{\SL}_2(\Ree)]/D\cong\wtil{\SL}_2(\Ree)$.  Notice that
 $[\Zee\times\{I\}]/D=R\cap S$, i.e.\ $(0,j(n))D=(-n,I)D$.  Hence, we have
\[
G/N_\Tr=RS/S=R/(R\cap S)\cong\Ree/\Zee\cong\Tee.
\]
This shows that $V$ can be smaller than the maximal vector subgroup
of $R/\wbar{[R,G]}$.

{\bf (ii)} If $\xi_1,\dots,\xi_m$ are rationally independent in $\Ree$,
then the map $\alp_m:\Zee\to\Tee^m$, $\alp_m(n)=(e^{in\xi_k})_{k=1}^m$, has dense
range.  Indeed, $\Zee^m+\Zee(\xi_k)_{k=1}^m$ is dense in $\Ree^m$, and projects
onto this range.
 
We consider
\[
G=(\Tee^m\times \wtil{\SL}_2(\Ree))/D\text{ where }D=\{(\alp_m(n),j(n)):n\in\Zee\}.
\]
Here $R=(\Tee^m\times\{I\})/D\cong \Tee^m$ is central, so $[R,G]=\{e\}$ and
$S=(\{e\}\times\wtil{\SL}_2(\Ree))/D\cong\wtil{\SL}_2(\Ree)$ is dense.
Hence $G$ is minimally almost periodic.  

Notice that $N_\cmap(G_d)=S$, and 
$G_d/N_\cmap(G_d)\cong \Tee_d^m/\alp_m(\Zee)$.

This is an easy modification of a well-known example, see \cite[14.5.10]{hilgertn}.
We have increased dimension of the torus to support Example \ref{ex:nLminap}, below.

{\bf (iii)} We let $\xi$  be irrational and let $H=\Ree^4$
with Heisenberg-type product
\begin{align*}
(x,&y,z,\zeta)(x',y',z',\zeta') \\
&=(x+x',y+y',z+z'+xy'-x'y,\zeta+\zeta'+\xi(xy'-x'y)).
\end{align*}
We see that $s=\begin{bmatrix} a & b \\ c & d\end{bmatrix}$ in the connected group 
$\GL_2(\Ree)_0$ acts on $H$ by
\[
s\cdot (x,y,z,\zeta)=(ax+by,cx+dy,(\det s)z,(\det s)\zeta).
\]
We let $G=H\rtimes \GL_2(\Ree)_0$.  Then $R=H\rtimes (\Ree^{>0} I)$
(recall that $R$ is connected so we use only the connected component of the 
centre of $\GL_2(\Ree)_0$),
$[R,G]=(\Ree^2\times \Ree(1,\xi))\rtimes\{I\}$ and
$S=\{0\}\rtimes\SL_2(\Ree)$.  Thus $G/N_\Tr\cong \Ree\times\Ree^{>0}\cong\Ree^2$.
This example and its successor were inspired by Rothman \cite[2.5]{rothman}.

{\bf (iii$'$)} The centre of $H$, above, is $Z(H)=\{0\}\times\Ree^2$, and is acted upon
trivially by $\SL_2(\Ree)$.  We let
$\wbar{H}=H/(\{0\}\times\Zee^2)\cong \Ree^2\times\Tee^2$.  This time
let $G=\wbar{H}\rtimes \SL_2(\Ree)$.  Here we find that $\wbar{[R,G]}=\wbar{H}\rtimes\{I\}$
and we get that $G$ is minimally almost periodic.  

Notice, in this case, that $[R,G]$ is not closed.  It may be computed that
$G_d/N_\cmap(G_d)$ is isomorphic to the discrete group
$\Tee^2_d$ modulo an irrational wind.
 \end{example}

\begin{remark}\label{rem:connected}
Suppose $G$ is connected, and let $\fL$ be a downward directed
collection of compact normal
subgroups $L$ for which $G/L$ is Lie, and $\bigcap_{L\in\fL}L=\{e\}$.  

{\bf (i)}  For $L\supseteq L'$ in $\fL$ we have $G/L\cong(G/L')/(L/L')$.  It follows that
$(G/L)/N_\Tr(G/L)=V\times K_L$ for the same vector group $V$ while $K_L$ is a quotient
of $K_{L'}$.  Hence $G/N_\Tr(G)=V\times K$ where $K$ is the inverse limit
of the compact Lie groups $K_L$.

{\bf (ii)} In \cite{shtern1} it is shown that $N_\Tr=N_\cmap$ has form similar 
to (\ref{eq:shtern}), and thus is connected.

{\bf (iii)} If each $G/L$ is solvable, then each $K_L$ is abelian, and hence so too is $K$.
Thus $N_\Tr$ contains the derived subgroup $[G,G]$, and hence its closure.
However, $G/\wbar{[G,G]}$ is abelian, so $\wbar{[G,G]}\subseteq N_\Tr$.
Hence $N_\Tr=\wbar{[G,G]}$.

We note that this captures aspects of Belti\c{t}\u{a} and Belti\c{t}\u{a} 
\cite[Theo.\ 2.8]{beltitab}, which were proved there by much different means.
\end{remark}

\begin{example}\label{ex:nLminap}
Let $\xi_1,\xi_2,\dots$ be rationally independent in $\Ree$, 
$\alp:\Zee\to\Tee^\En$ be given by $\alp(n)=(e^{i n \xi_k})_{k=1}^\infty$.
Any non-empty open set in $\Tee^\En$ meets $\alp(\Zee)$, thanks to
Example \ref{ex:Lie} (ii), so $\wbar{\alp(\Zee)}=\Tee^\En$.  We then consider
\[
G=[\Tee^\En\times\wtil{\SL}_2(\Ree)]/D\text{ where }D=\{(\alp(n),j(n)):n\in\Zee\}.
\]
By the last remark, the family $\fL=\{\Tee^{\{m+1,m+2,\dots\}}:m\in\En\}$, 
with each Lie quotient of the form given in Example \ref{ex:Lie} (ii),
shows that $G$ is minimally almost periodic.
\end{example}

We have only one interesting thing to say about the trace kernels of
non-connected Lie groups.

\begin{proposition}
Let $G$ be an almost connected Lie group.  Then $N_\Tr(G)=N_\Tr(G_0)$.
In particular, $N_\Tr(G)$ is connected.
\end{proposition}

\begin{proof}
We have that $G_0$ is open.  Since $\tpos(G)|_{G_0}\subseteq\tpos(G_0)$ we see that
$N_\Tr(G_0)\subseteq N_\Tr(G)$.  To see the converse, suppose 
$r\in G_0\setminus N_\Tr(G_0)$.  Then there is a $u$ in $\tpos(G_0)$ for which $u(r)\not=1$.
Hence if $v=\frac{1}{2}(u+\bar{u})$ then $v\in\tpos(G_0)$ with $-1\leq v(r)<1$.  Define 
\[
\dot{v}(s)=\begin{cases} 
\frac{1}{|G/G_0|}\sum_{tG_0\in G/G_0}u(t^{-1}st) &\text{if }s\in G_0 \\
0 &\text{otherwise.}\end{cases}
\]
Then $\dot{v}$ is well-defined, and continuous since $G_0$ is open, and 
thus in $\tpos(G)$, with $\dot{v}(r)<1$.  
\end{proof}

\begin{example}\label{ex:finiteindex}
Let $\alp$ be an inner automorphism on $\SL_2(\Ree)$, as given in
Example \ref{ex:gammaF} (i), above.  Then the last proposition
shows for $G=\SL_2(\Ree)\rtimes\langle\alp\rangle$ that
$N_\Tr=\SL_2(\Ree)\rtimes\{e\}$, and hence 
is of finite index in $G$.
\end{example}

\begin{remark}
Since an almost connected compactly generated group $G$ is compactly generated,
we know that $G/N_\Tr$ is almost connected and in $[\csin]$ and hence
of the form $V\rtimes K$ where the action of $K$ on $V$ factors through a finite group.
The vector group $V$, being connected, is a vector quotient
of $G_0$, which is also the same vector quotient of any quotient $G_0/L$
for any compact co-Lie subgroup $L$ of $G_0$.

Hence, much like (\ref{eq:tracon}) we have that
\[
\tpos(G)\cong\tpos(V\rtimes K)=\wbar{\mathrm{conv}}
\left[\pos_1^K(V)\otimes\left\{\frac{1}{d_\pi}\chi_\pi:\pi\in\what{G}\right\}\right]
\]
where $\pos_1^K(V)=\{u\in\pos_1(V):u(k\cdot v)=u(v)\text{ for }v\text{ in }V\text{ and }k
\text{ in }K\}$, which is isomorphic to the space of probabilities
on the dual group, averaged over (finite) orbits of $K$, $\mathrm{Prob}^K(\what{V})$.
\end{remark}

\section{Reduced traces}\label{sec:redtra}

We let $\lam=\lam_G:G\to\Un(\bl^2(G))$ denote the left regular representation
\[
\pos_\lam(G)=\{\langle\lam(\cdot)f,f\rangle:f\in\bl^2(G)\}.
\]
This space is naturally isomorphic to the space of normal positive functionals
on the group von Neumann algebra $\vn(G)=\lam(G)''$, as this algbera is
in standard form.  We let
\[
\pos_r(G)=\wbar{\pos_\lam(G)}^{w*}=\wbar{\pos_\lam(G)}^{uc}\subseteq\pos(G)
\]
where $w^*$ represents the weak* topology of $\cstar(G)$ and $uc$ the topology
of uniform convergence on compact sets; see \cite[13.5.2]{dixmier}.
Then $\pos_r(G)$ is naturally
identified with the space of positive linear functionals on the reduced
C*-algebra $\cstar_r(G)$ and consists of all positive definite matrix coefficients 
of representations which are weakly contained in $\lam$.  

We recall Hulanicki's Theorem: 
 \[
 G\text{ is amenable}\;\Leftrightarrow\;
 1\in\pos_r(G)\;\Leftrightarrow\;
 \pos_r(G)=\pos(G)\;\Leftrightarrow\;
 \cstar_r(G)\cong\cstar(G).
 \]
Let us review two properties of these spaces, which are familiar to specialists.

\begin{lemma}\label{lem:redpd}
\begin{itemize}
\item[(i)] If $H$ is a closed subgroup of $G$ then $\pos_r(G)|_H\subseteq
\pos_r(H)$. 
\item[(ii)] If $N$ is an amenable closed normal sugbroup of $G$ then
$\pos_r(G/N)\circ q_N\subseteq \pos_r(G)$.
\end{itemize}
\end{lemma}

\begin{proof}
{\bf (i)} We have that $\pos_\lam(G)|_H\subseteq \pos_\lam(H)$. 
Indeed, $\pos_\lam(G)$ is the closure of the compactly supported
positive definite functions; see \cite[13.8.6]{dixmier} or \cite[7.2.5]{pedersen}.

{\bf (ii)} Amenability provides the weak containment $1_N\prec\lam_N$.
Then Fell's continuity of induction provides
\[
\lam_{G/N}\circ q_N=\ind_N^G 1_N\prec \ind_N^G\lam_N=\lam_G.
\]
Thus $\pos_\lam(G/N)\circ q_N\subseteq \pos_r(G)$.
\end{proof}

The following is surely known, admitting the same proof as the fact
that totally disconnected compact groups are pro-finite, but is necessary
for the main result of this section.

\begin{lemma}\label{lem:tdsin}
If $G$ in $[\csin]$ is totally disconnected, then $G$ is pro-discrete, i.e.\
there is a base at the identity consisting of compact open normal subgroups.
\end{lemma}

\begin{proof}
Let $U$ be a conjugation-invariant neighbourhood of the identity.  Then
$U$ contains a compact open subgroup $K$ which, in turn, contains another
conjugation invariant neighbourhood $V$, which contains a compact open subgroup
$L$.  Then the subgroup generated by $\bigcup_{x\in G}xLx^{-1}$ is open, normal
and contained in $K$.
\end{proof}

We let the {\it reduced traces} be given by
\[
\tpos_r(G)=\tpos(G)\cap\pos_r(G).
\]
These are the tracial states on $\cstar_r(G)$.  We then
consider the {\it reduced tracial kernel}
\[
N^r_\Tr=N^r_\Tr(G)=\bigcap_{u\in\tpos_r(G)}N_u
\]
which we deem to be all of $G$ if $\tpos_r(G)=\varnothing$.

The following was the main result of an earlier version of our work
\cite{forrestsw}.

\begin{theorem}\label{theo:redtra}
If $G$ is compactly generated then the following are equivalent:
\begin{itemize}
\item[(i)] $\tpos_r(G)\not=\varnothing$;
\item[(ii)] $\fN_\csin$ admits an amenable element; 
\item[(iii)] $G$ admits an open normal amenable subgroup; and
\item[(iv)] $N_\Tr^r$ is amenable.
\end{itemize}
\end{theorem}

\begin{proof}
{\bf (i) $\Rightarrow$ (ii)} If $u\in\tpos_r(G)$, then Lemma \ref{lem:redpd} (i) gives that
$1=u|_{N_u}\in\pos_\lam(N_u)$ so $N_u$ is amenable.  Proposition \ref{prop:comgen}
provides that $N_u\in\fN_\csin$.

{\bf (ii) $\Rightarrow$ (iii)} 
If $N$ in $\fN_\csin$ is amenable, then the Fruedenthal-Weil Theorem provides
that $(G/N)_0$ is amenable, and $\wtil{N}=q_N^{-1}((G/N)_0)$ is also amenable with
$G/\wtil{N}\cong (G/N)/(G/N)_0$ a totally disconnected $\csin$-group.
Lemma \ref{lem:tdsin} provides a compact open normal subgroup
$M$ in $G/\wtil{N}$, so $q_{\wtil{N}}^{-1}(M)$ is open, normal and amenable.

{\bf (iii) $\Rightarrow$ (iv)}  Let $N$ be an open, normal amenable subgroup of $G$.
Then Lemma \ref{lem:redpd} (ii) shows that 
$1_N=\langle\lam_{G/N}\circ q_N(\cdot)\del_{eN},\del_{eN}\rangle\in\tpos_r(G)$.
Hence $N_\Tr\subseteq N$ and is thus amenable.

{\bf (iv) $\Rightarrow$ (i)} If $N^r_\Tr=G$, then $G$ is amenable, so $\tpos_r(G)=\tpos(G)$.
If $N^r_\Tr\not= G$, then $\tpos_r(G)\not=\varnothing$.
\end{proof}

Our proof of (ii) $\Rightarrow$ (iii) $\Rightarrow$ (iv) $\Rightarrow$ (i), holds in absence of
the assumption of compact generation.  Our proof of (iii) $\Rightarrow$ (i)
generalizes and conceptually simplifies \cite[Prop.\ 1.6]{paschkes}.

The {\it amenable radical}, $AR(G)$ of a locally compact group $G$ is the largest
amenable closed normal subgroup.  
For second countable $G$, the existence is given in \cite[4.1.2]{zimmer}.
Let us briefly outline the main steps to show its existence more generally.

If $M,N$ are amenable normal subgroups then
$M/(M\cap N)$ embeds continuously and densely into $\wbar{MN}/N$,
showing that $\wbar{MN}$ is amenable.  Let 
\[
AR(G)=\wbar{\bigcup\{ N:\text{ closed normal amenable
subgroup}\}}\subseteq G.
\]
Any cluster point of the net of $N$-invariant means, $\wtil{M}_N(\psi)=M_N(\psi|_N)$
for bounded left uniformly
continuous $\psi$ on $AR(G)$, indexed over increasing closed normal
subgroups, admits an invariant mean as any cluster point.
This is less delicate if some such a subgroup is open.
Condition (iii) above is clearly equivalent to
\begin{itemize}
\item[(iii$'$)] {\it $AR(G)$ is open.}
\end{itemize}

\begin{remark}\label{rem:connliert}
Suppose $G$ is a connected Lie group, and recall the notation of Proposition
\ref{prop:connlie}.  We have that $AR(G)=RS_cZ(S_{nc})$.  Hence it is immediate
that 
\[
\tpos_r(G)\not=\varnothing\quad\Leftrightarrow\quad AR(G)=G
\quad\Leftrightarrow\quad G/R\text{ is compact.}
\]
In this case (\ref{eq:shtern}) provides that
 $\tpos_r(G)=\tpos(G)=\tpos(G/\wbar{[G,R]})\circ q_{\wbar{[G,R]}}$.
Given \cite[10.25, 10.28 \& 10.29]{hofmannm}
and Remark \ref{rem:connected} (ii), the same holds for any connected group.
\end{remark}

We note the following beautiful result.

\begin{theorem} \label{theo:kennedyr} {\rm (Kennedy and Raum \cite{kennedyr})}
We have that $\tpos_r(G)\not=\varnothing$ if and only if $AR(G)$ is open in $G$.
Furthermore, in this case we have
\begin{equation}\label{eq:redtraces}
\tpos_r(G)\cong\left\{u\in\tpos(AR(G)):
\begin{matrix} u(trt^{-1})=u(r)\text{ for }r \\ \text{in }AR(G)\text{ and }
t\text{ in }G\end{matrix}\right\}.
\end{equation}
\end{theorem}

Indeed, the authors in \cite{kennedyr} show that any element of $\lam(\fC_c(G))$
supported in $G\setminus AR(G)$ is annihilated by each reduced trace.
Hence traces are supported in $AR(G)$.  A trace on $AR(G)$
admits a trivial extension (i.e.\ $0$ outside of $AR(G)$) to a reduced
trace on $G$ exactly when it is conjugation invariant by actions from $G$.

\begin{remark}
We note that our Theorem \ref{theo:redtra} gives, for compactly generated groups,
a much simpler proof of the first statement of Theorem \ref{theo:kennedyr}.
Combined with Proposition \ref{prop:connlie} we are now in a position to have a clear understanding 
of the nature of traces in the case for connected groups; see Remark \ref{rem:connliert}, above.

It is worth noting that we proved our main result in \cite{forrestsw} before we were aware of the work of \cite{kennedyr}.
Happily, we can combine both methods to gain a deeper understanding of the structure
of reduced traces.
\end{remark}

As in (\ref{eq:tracered}), we see that if $AR(G)$ is open, then (\ref{eq:redtraces})
becomes
\begin{equation}\label{eq:redtraces1}
\tpos_r(G)\cong\left\{u\in\tpos(AR(G)/N^r_\Tr):
\begin{matrix} u(trt^{-1}N^r_\Tr)=u(rN^r_\Tr)\text{ for} \\ r\text{ in }AR(G)\text{ and }
t\text{ in }G\end{matrix}\right\}.
\end{equation}
We summarize some immediate consequences of this observation.

\begin{proposition}\label{prop:urt}
If $AR(G)/N^r_\Tr$ is finite, then $\tpos_r(G)$ is finite dimensional.
If $AR(G)=N^r_\Tr$ and is open, we will get unique reduced trace.
\end{proposition}

By way of contrast,
with aid of Lemma \ref{lem:redpd} (ii) and Theorem \ref{theo:redtra},
we may recycle the proof of Proposition
\ref{prop:trasin} (ii) to see the following:

\begin{proposition}
If  $G$ admits a non-open amenable element of $\fN_\csin$, 
then $\tpos_r(G)$ is infinite dimensional.
This happens, in particular, if $G$ is compactly generated and
$N_\Tr^r$ is not open.
\end{proposition}

We shall use the following to help devise and inspect examples:

\begin{theorem}\label{theo:AGamma}
Let $\Gamma$ be a finitely generated group which acts irreducibly on a
non-discrete abelian group $A$,  in the sense that
$A$ admits no non-trivial $\Gamma$-invariant closed subgroups.  Then 
for $G=A\rtimes \Gamma$ we have that $N_\Tr^r\cap (A\rtimes\{e\})$ is either
$A\rtimes\{e\}$ or the trivial subgroup.  The latter case
is equivalent to the action of $\Gamma$ factoring through
a compact group of automorphisms acting continuously on $A$. 
\end{theorem}

\begin{proof}
The irreducibility condition tells us that $N_\Tr^r\cap A$ is either $A$ or $\{e\}$.  Let us
suppose the latter.  

Since $A_0$ is characteristic in $A$, the irreducibility condition entails either
that $A$ is connected or totally disconnected.  In the latter case, let $K$ be any compact open subgroup of $A$.  Then $\Gamma(K)=A$ so $G=\langle (K\rtimes\{e\})\cup
(\{e\}\rtimes\Gamma)\rangle$ is compactly generated.  Hence we see generally
that $G$ is compactly generated.   Since $\{e\}=N_\Tr^r\supseteq N_\Tr$, Corollary
\ref{cor:comgen} tells us that $G\in[\csin]$, and hence $A\in[\csin]$ 
and admits a neighbourhood base at the identity consisting of $\Gamma$-invariant sets.  

Notice, in particular, that
each $\Gamma$-orbit in $A$ is relatively compact.  Indeed, let $K$ be either a compact
generating set (in the case that $A$
is connected), or an open compact subgroup (in the case that $A$ is totally disconnected),
contained in a relatively compact $\Gamma$-invariant neighbourhood of the identity.
Then any element of $\Gamma(K^n)$, for $n$ in $\En$ has relatively compact 
$\Gamma$-orbit.

The conditions gathered in the last two paragraphs imply that the image of $\Gamma$
in the automorphism group of $A$ (with compact open topology)
is relatively compact, thanks to the Ascoli Theorem
of Grosser and Moskowitz, \cite[Theo.\ 4.1]{grosserm}.

Conversely, suppose 
there is a compact group $\Sigma$ of automorphisms on $A$ which
contains the image of $\Gamma$.  For $\chi\in\what{A}\setminus\{1\}$, 
we average over the normalized
Haar measure to get $u=\int_\Sigma\chi\circ\sigma\,d\sigma$, which is  
a $\Gamma$-invariant trace on $A$.
For this $u$ we have $N_u=\bigcap_{\sigma\in\Sigma}\ker(\chi\circ\sigma)$, which, thanks
to the irreducibility assumption, is $\{e\}$.
\end{proof}

A semi-direct product with a compact  group is simpler.

\begin{proposition}\label{prop:KGamma}
Let $\Gamma$ be a discrete group which acts on a compact group $K$.
Then for $G=K\rtimes \Gamma$, $N_\Tr^r\subsetneq K$ if and only if 
$X_K=\{\chi_\pi:\pi\in\what{K}\}$ admits a non-trivial finite orbit for the adjoint action.
\end{proposition}

\begin{proof}
If $\sigma$ is a continuous automorphism of $K$ and $\pi\in\what{K}$, then it
is evident that $\pi\circ\sigma\in\what{K}$ with $\chi_\pi\circ\sigma=\chi_{\pi\circ\sigma}$.

Since $\tpos(K)=\wbar{\mathrm{conv}}X_K$, each element $u$ in $\tpos(K)$
admits decomposition $u=\sum_{\pi\in\what{K}}\breve{u}(\pi)\chi_\pi$ where
$\breve{u}\in\mathrm{Prob}(\what{K})$.  Then $u$ is $\Gamma$-invariant
if and only if each element of $\{\chi_\pi\in X_K:\breve{u}(\pi)>0\}$ admits finite
orbit. If such $u$ has $\breve{u}(\pi)>0$ for $\pi\not=1$, then
$N_u\subseteq \bigcap_{\sigma\in\Gamma}\ker(\pi\circ\sigma)\subsetneq K$.
\end{proof}

\begin{remark}
Theorem \ref{theo:AGamma} and Proposition \ref{prop:KGamma} also hold
if $G$ is an extension:  $H\to G\to \Gamma$, where $H=A$ or $K$.  However,
these results are intended to aid in the building of examples, for which
the semi-direct product formulation is adequate.
\end{remark}

\begin{example}\label{ex:redtraces}
Here we shall always write $G=A\rtimes\Gamma$ or $K\rtimes\Gamma$, 
and $N_\Tr^r=N_\Tr^r(G)$.

{\bf (i)} Let $\Gamma$ be a non-commutaitve free group, which is dense inside
$\SO(d)$, and $A=\Ree^d$.  Then $N_\Tr^r=\{e\}$.
We have that $AR(G)=A\rtimes\{I\}$, thanks to \cite{powers}.

{\bf (ii)} Let $M$ in $\GL_2(\Ree)$ have non-real eigenvalues, of modulus not $1$.
Let $q:\Gamma=\Fr_2\to\Zee^2$ be the quotient map onto the abelianization, and
let $\eta:\Zee^2\to\Ree$ be given by $\eta(m,n)=m+\xi n$ where $\xi$ is irrational.
Let $\Gamma$ act on $A=\Ree^2$ by $\sigma\cdot x=\exp(\eta\circ q(\sigma)M)x$.
This is an irreducible action with non-relatively compact non-trivial $\Gamma$-orbits. 
Hence $N_\Tr^r=A\rtimes\{I\}$.  Notice that  $\what{A}$ will admit no non-trivial 
$\Gamma$-invariant probability measure.  However, since this action factors through an 
abelian group,  $\what{A}$ will support a $\Gamma$-invariant mean on the bounded 
uniformly continuous functions.

As in (i), above, $AR(G)=A\rtimes\{I\}$.

{\bf (iii)} Let $\Gamma=\SL_d(\Zee)$ act on $A=\Tee^d$.  
Then Proposition \ref{prop:KGamma} shows that $N_\Tr^r=A\rtimes\{I\}$.

It is shown in \cite{bekkacdlh} that $\mathrm{PSL}_d(\Zee)$
admits unique reduced trace.  This group is $\Gamma$ when $d$ is odd and
$\Gamma/\langle-I\rangle$ when $d$ is even.
It follows that
\[
AR(G)=\begin{cases}A\rtimes\{I\} &\text{if }d\text{ is odd, and} \\
A\rtimes\langle-I\rangle &\text{if }d\text{ is even.}\end{cases}
\]

{\bf (iii$'$)}  Let $\Gamma=\SL_d(\Zee)$ act on $A=\Ree^d$.  
Then $A$ is not irreducible for the action of $\Gamma$, but
$G$ is compactly generated.  Hence, as in the proof of Theorem \ref{theo:AGamma},
 if we had $N^r_\Tr\cap (A\rtimes\{I\})=\{(0,I)\}$, then $A$ would have a base of
$\Gamma$-invariant neighbourhoods at the identity, which is clearly false.
The only other closed $\Gamma$-invariant subgroups are lattices
$L\rtimes\{I\}$, and $A\rtimes\{I\}$.  If we had $N^r_\Tr=L\rtimes\{I\}$, then
(\ref{eq:tracered}) shows that we would violate (iii).  Hence $N^r_\Tr=A\rtimes{I}$.

In this example we may instead take $\Gamma=\SL_d(\Que)$.  Though
this group is not finitely generated, we just observed a subgroup which
allows no non-trivial invariant traces  on $A$.  Since $\mathrm{PSL}_d(\Que)
=\Gamma/\langle -I\rangle$ is simple, we get descriptions of $AR(G)$, similar
to those in (iii) above.

{\bf (iii$''$)} Consider a free group $\Gamma=\Fr_2$ embedded into
$S=\SL_2(\Zee)$ with finite index.
Again for $A=\Ree^2$ or $A=\Tee^2$ we see that $N^r_\Tr=A\rtimes\{e\}$.
As in (i), above, $AR(G)=A\rtimes\{I\}$.

{\bf (iii$'''$)} We let $\Gamma=\Fr_2$ as above, act on $A=H$ where
$H$ is from Example \ref{ex:Lie} (ii).  Here $A$ is nilpotent with closed
derived subgroup and $A/A'\cong\Ree^2$. Then it follows (\ref{eq:shtern}) that 
$\tpos(A)=\tpos(A/A')\circ q_{A'}$.  Thus this reduces to (iii$''$), above,
and we see that $N_\Tr^r=A\rtimes\{e\}$.

{\bf (iv)} Let $\Gamma$ be any infinite discrete group and let it act on 
a product group $K=L^\Gamma$, where $L$ is compact, by shifting index.
Each non-trivial member of $\what{K}$ is a finite Kroenecker product 
$\pi_1\times\dots\times\pi_n$
where each element $\pi_j$ acts on a distinct copy of $L$ in $K$.
Then non-trivial adjoint orbits in $X_K$ are infinite,
so Proposition \ref{prop:KGamma} shows that $N_\Tr^r=K\rtimes\{e\}$.

{\bf (v)} Let $K$ be a semi-simple compact Lie group,
so $K=\prod_{i\in 1}^nS_i/D$ where each $S_i$ is a simple compact Lie
group and $D$ is a finite central subgroup.
By looking at automorphisms of the associated Lie algebra, we see that
$\mathrm{Aut}(K)$ is a subgroup of $I\rtimes P$, where 
$I\cong\prod_{i\in 1}^n(S_i/Z_i)$ is the group of inner automorphisms where
each $Z_i$ is the centre of $S_i$, and $P$ is a discrete group of permutations of 
isomorphic constituents $S_i$.  Elements of $I$ fix elements of $X_K$,
so a group $\Gamma\subset \mathrm{Aut}(K)$, acts on $X_K$ as does $\Gamma/
(\Gamma\cap I)\subseteq P$, and hence has finite orbits.  It follows that $N_\Tr^r=\{(I,e)\}$.

{\bf (vi)}  We let $S=\SL_2(\Zee[1/p])$ act on $\mat_2(\Que_p)$
by $\sigma\cdot x=\sigma x \sigma^T$, which leaves invariant the subspace
$L=\Que_p\begin{bmatrix} 0 & 1 \\ -1 & 0\end{bmatrix}$, and
induces an action of
$\Gamma=\mathrm{PSL}_2(\Zee[1/p])\cong S/\langle-I\rangle$
on $A=\mat_2(\Que_p)/L$.  
Here, no non-trivial $\Gamma$-orbit in $A$ is relatively compact,
so $N_\Tr^r=A\times\{e\}$.  Since $\Gamma$ admits unique reduced trace \cite{delaharpe,bekkacdlh1},
we see that $AR(G)=A\times\{e\}$, as well.

{\bf (vii)} Let $\Gamma=\SL_d(\Zee)$ act on $A=\Que_p^d$.  This fails
the irreduciblility hypothesis of Theorem \ref{theo:AGamma}, as subgroups
$p^n\Oh_p^d$ are each $\Gamma$-invariant.  Here the action factors through
$\SL_d(\Oh_p)$, so conclusion (b) of the theorem still holds.
\end{example}

Recall that we say that $G$ if {\it C*-simple} of $\cstar_r(G)$ is simple.  This is
well-known to imply that $AR(G)=\{e\}$.  [Indeed, the method of Lemma \ref{lem:redpd}
(ii) shows that $\cstar_r(G/AR(G))$ is a proper quotient of $\cstar_r(G)$ if 
$AR(G)\not=\{e\}$.]
In each example above, $\cstar_r(G)$ admits as a proper quotient
$\cstar_r(\Gamma)$, so none of these examples give a C*-simple group.
Le Boudec \cite{leboudec} produces examples of discrete groups with $AR(G)=\{e\}$,
but are not C*-simple.

We saw that connected groups admitting non-compact semi-simple quotients
have non-open amenable radical.  The same extends to certain semi-simple
algebraic groups over non-discrete local fields,
such as $\SL_n(\kay)$ and $\Sp_n(\kay)$ for general $\kay$, and
$\SO_n(\kay)$ when $\kay\not=\Ree$.  Each of these is 
admits only the trivial group as a discrete quotient.
Let us next consider further examples that admit no reduced traces, but many
quotients.

\begin{example}\label{ex:suzuki}
We recall the notation of Example \ref{ex:gammaF}.  It is straightforward
to see that  {\it $AR(G)$ is
open if and only if $\{e\}\rtimes F_n\subseteq AR(\Gamma_n\rtimes F_n)$
for all $n\geq m$ for some $m$.} We produce two examples where 
$\tpos_r(G)=\varnothing$.

{\bf (i)} As we saw in Example \ref{ex:redtraces} (i), $AR(\SL_2(\Zee))=\langle-I\rangle$.
Let $\alp$ and $G$ be as in  Example \ref{ex:gammaF} (ii).  
We have that $AR(\SL_2(\Zee)\rtimes
\langle\alp\rangle)=\langle -I\rangle\rtimes\{e\}$.  

{\bf (ii)} The following example is from Suzuki \cite{suzuki}.
Each $G_n=\Zee\ast(\Zee/k_n\Zee)$ with $k_n\geq 2$ has $AR(G_n)=\{e\}$,
by \cite{paschkes}.  If $\Gamma_n$ is the normalizer of $\Zee$ in $G_n$, then
$G_n=\Gamma_n\rtimes F_n$ where $F_n=\Zee/k_n\Zee$.
Notice that this group is residually finite, thanks to Gruenberg \cite{gruenberg}.

That $\tpos_r(G)=\varnothing$, in this case,
was shown by different means in \cite{forrestsw}.
The goal of this example was to have a non-discrete
C*-simple group with unique reduced trace.  
Example \ref{ex:redtraces}, above, gives non-C*-simple groups with unique trace.
\end{example}



\section{Amenable Traces}\label{sec:ametra}

In the spirit of Brown \cite[3.1.6]{brown} (really Kirchberg \cite[Prop.\ 3.2]{kirchberg}) 
we shall say that
a tracial state $\tau$ on $\fA$ is {\it amenable} (or {\it liftable}) provided 
that $a\otimes b\mapsto \tau(ab)$
extends to a state on the minimal tensor product $\fA\otimes_{\min}\fA^\op$.

To study amenable traces on groups we begin by recording the well-known facts that
$\cstar(G)$ and $\cstar_r(G)$ are {\it symmetric},
i.e.\ each isomorphic to its opposite algebra.  This allows us to establish concepts
and notation.

Let $\pi:G\to\Un(\fH)$ be a unitary representation.  We let its
{\em C*-algebra} and {\em positive definite cone} be given by
\[
\cstar_\pi=\wbar{\pi(\bl^1(G))}\quad\text{ and }\quad\pos_\pi
=\mathrm{cone}\{\langle\pi(\cdot)\xi,\xi\rangle:\xi\in\fH\}.
\]
We let $\bar{\pi}$ denote the contragradient representation, given by
$\langle\bar{\pi}(s)\xi^*,\eta^*\rangle = \langle \eta,\pi(s)\xi\rangle$,
where $\xi^*(\zeta)=\langle \zeta,\xi\rangle$ for $\zeta$ in $\fH$.
We let 
\[
\varpi=\bigoplus_{u\in\pos_1(G)}\pi_u^\infty
\] 
 be the direct sum of GNS representations
from all states, with infinite ampliation; so, for example $\varpi^\infty
\cong\varpi$.   Then $\cstar(G)\cong\cstar_\varpi$.

\begin{proposition}\label{prop:sym}
We have an isomorphism $(\cstar_\pi)^\op\cong\cstar_{\bar{\pi}}$.
Hence if we have unitary equivalence, $\pi\cong\bar{\pi}$,
then $\cstar_{\bar{\pi}}$ is symmetric.

In particular, each of $\cstar(G)$ and $\cstar_r(G)$ are symmetric.
\end{proposition}

\begin{proof}
Let $F_\pi=\spn\pos_\pi=\spn\{\langle\pi(\cdot)\xi,\eta\rangle:\xi,\eta\in\fH\}$,
where the second equality follows from polarization identity.  We have 
$\wbar{\langle\pi(\cdot)\xi,\eta\rangle}=\langle\bar{\pi}(\cdot)\xi^*,\eta^*\rangle$.
This shows that if $\pi\cong\bar{\pi}$, then $F_\pi$, equivalently
$\pos_\pi$, is closed under conjugation.

For $f\in\bl^1(G)$ define
$\til{f}(s)=\Del(s^{-1})f(s^{-1})$ for almost every $s$ in $G$, so $f\mapsto\til{f}$
is an anti-homomorphism.  Then for $\xi,\eta$ in $\fH$ we have
\[
\langle\pi(\til{f})\xi,\eta\rangle=\int_G f(s)\langle\pi(s^{-1})\xi,\eta\rangle\,ds
=\langle \bar{\pi}(f)\xi^*,\eta^*\rangle.
\]
Hence $\|\pi(\til{f})\|=\|\bar{\pi}(f)\|$ for each $f$ in $\bl^1(G)$, so
$f\mapsto\til{f}$ induces an anti-isomorphism from $\cstar_\pi$ onto
$\cstar_{\bar{\pi}}$.  Furthermore, if $F_\pi$ is closed under conjugation,
then $\|\bar{\pi}(f)\|=\|\pi(f)\|$ for each $f$ in $\bl^1(G)$, so $\cstar_\pi$
is symmetric.

We have unitary equivalences
\[
\varpi=\bigoplus_{u\in\pos_1(G)}\pi_u^\infty=\bigoplus_{u\in\pos_1(G)}\pi_{\bar{u}}^\infty
\cong \bigoplus_{u\in\pos_1(G)}\bar{\pi}_u^\infty=\wbar{\varpi}.
\]
Meanwhile, we have for $h$ in $\bl^2(G)$ that
$\wbar{\langle\lam(\cdot)h,h\rangle}=\langle\lam(\cdot)\wbar{h},\wbar{h}\rangle$.
These observations provide symmetry for $\cstar(G)$ and $\cstar_r(G)$.
\end{proof}

\begin{remark}
The algebras $\cstar_\pi$ need not be symmetic.
If $\fA$ is a non-symmetric unital C*-algebra, let $\pi:\fA\to\fB(\fH)$ be a faithful
representation and $\Gamma=\Un(\fA)_d$, the discretized group of unitaries.
Then $\fA\cong\pi(\fA)=\cstar_{\pi|_\Gamma}$.
\end{remark}

If $u\in\tpos(G)$, define $\til{u}:G\times G\to\Cee$ by
\[
\til{u}(s,t)=u(st^{-1}).
\]
Notice that if if $(s_1,t_1),\dots,(s_n,t_n)\in
G\times G$ then for $i,j=1,\dots, n$ we have
\[
\til{u}(s_i^{-1}s_j,t_i^{-1}t_j)=u(s_i^{-1}s_jt_j^{-1}t_i)=u(t_is_i^{-1}(t_js_j^{-1})^{-1})
\]
from which it follows that $\til{u}\in\pos_1(G\times G)$.

The universal representation $\varpi:G\to\Un(\fH_\varpi)$ has Kronecker product given by
$\varpi\times\varpi:G\times G\to\Un(\fH_\varpi\otimes\fH_\varpi)$.  We let
\[
\pos_{\min}(G\times G)=\wbar{\pos_{\varpi\times\varpi}}^{w*}=
\wbar{\pos_{\varpi\times\varpi}}^{uc}\subseteq\pos(G\times G).
\]
Notice that $\pos_{\min}(G\times G)$ corresponds to the positive linear functionals
on $\cstar(G)\otimes_{\min}\cstar(G)$, and hence comprises all of the positive
definite matrix coefficients of representations weakly contained in $\varpi\times\varpi$.
Furthermore
\begin{equation}\label{eq:krocon}
\pos_{\min}(G\times G)=\wbar{\mathrm{cone}\{u\times v:u,v\in\pos(G)\}}^{uc}
\end{equation}
where each $u\times v(s,t)=u(s)v(t)$.

Since Proposition \ref{prop:sym} provides that $\cstar(G)\cong\cstar(G)^\op$ we
can define the set of {\em amenable traces} on $G$ by
\[
\tpos_{\mathrm{am}}(G)=\{u\in\tpos(G):\til{u}\in\pos_{\min}(G\times G)\}.
\]
Accordingly, we define the {\em amenable tracial kernel} by
\[
N_\amTr=N_\amTr(G)=\bigcap_{u\in\tpos_{\mathrm{am}}(G)}N_u
\]
and say that $G$ is {\em amenably tracially separated} if $N_\amTr=\{e\}$.

\begin{remark}\label{rem:amkmap}
If $u\in\tpos(G)$, then $\til{u}$ corresponds to a state on
$\cstar_{\pi_u}\otimes_{\max}(\cstar_{\pi_u})^\op$.  If $\cstar_{\pi_u}$ is nuclear,
then we employ Proposition \ref{prop:sym} along the way to see that
\[
\cstar_{\pi_u}\otimes_{\max}(\cstar_{\pi_u})^\op
\cong\cstar_{\pi_u}\otimes_{\min}\cstar_{\pi_{\bar{u}}}\cong
\cstar_{\pi_u\times \pi_{\bar{u}}}
\]
so we get the following  weak containments: 
\[
\pi_{\til{u}}\prec \pi_u\times \pi_{\bar{u}}\prec\varpi\times\varpi.
\]
That is,\ $u\in\tpos_{\mathrm{am}}(G)$.

{\bf (i)} Characters $\chi_\pi$ of finite dimensional representations
$\pi$ show that $N_\amTr\subseteq N_\cmap$. 

{\bf (ii)} If $G$ is either amenable, almost connected or 
type I, then $\cstar(G)$, and hence any quotient, is nuclear;
see the survey \cite{paterson} and references therein.
Hence $\tpos(G)=\tpos_{\mathrm{am}}(G)$, and $N_\Tr=N_\amTr$. 

 It has recently been shown by Bekka and Echterhoff
\cite{bekkae} that any algebraic group over a local field, $\alg(\kay)$, is type I.

{\bf (iii)} If $G$ is almost connected, then it is compactly generated so 
Proposition \ref{prop:comgen} and the result of Grosser and Moskowitz 
\cite[(2.9)]{grosserm2} 
show that $G/N_\amTr\in[\cmap]$, and hence $N_\amTr=N_\cmap$.
\end{remark}

\subsection{Property (T)}
The the next result stems from  Ozawa \cite[Theo.\ 7.2]{ozawa}, where it is shown 
that each discrete group with both property (T) and the factorization property (see 
defintion before Theorem \ref{theo:facpro}, below)  is residually finite.

\begin{theorem}\label{theo:amkT}
If $G$ has property (T), then  $N_\amTr=N_\cmap$.  
\end{theorem}

\begin{proof}
That $N_\amTr\subseteq N_\cmap$ is given in the last remark.

If $u\in\tpos_{\mathrm{am}}(G)$, then since $\varpi^\infty=\varpi$,  we can find
a net of unit vectors $(\xi_i)$ in $\fH_\varpi\otimes\fH_\varpi$ for which
\[
\til{u}=uc\text{-}\lim_i\langle\varpi\times\varpi(\cdot,\cdot)\xi_i,\xi_i\rangle.
\]
Restricting to 
the diagonal subgroup $G_D=\{(s,s):s\in G\}\cong G$ we see that
\[
1=\til{u}|_{G_D}=uc\text{-}\lim_i\langle\varpi\otimes\varpi(\cdot)\xi_i,\xi_i\rangle
\]
and hence
\[
0=uc\text{-}\lim_i\|\varpi\otimes\varpi(\cdot)\xi_i-\xi_i\|
\]
The assumption of property (T) allows us to find unit vectors $(\xi'_i)$ tending
asymptotically to $(\xi_i)$ with 
\begin{equation}\label{eq:fp}
1=\langle\varpi\otimes\varpi(\cdot)\xi'_i,\xi'_i\rangle\text{ for each }i.  
\end{equation}
Let $P$ be the projection onto the almost periodic part of $\fH_\varpi$,
so by Berglund and Rosenblatt \cite[1.14]{bergelsonr},
$P\otimes P$ is the projection onto the almost periodic part of
$\fH_\varpi\otimes\fH_\varpi$.
Then (\ref{eq:fp}) tells us that $\xi'_i=P\otimes P\xi'_i$ for each $i$.  We have
\[
\til{u}=uc\text{-}\lim_i\langle\varpi\times\varpi(\cdot,\cdot)\xi'_i,\xi'_i\rangle
\text{ so }u=uc\text{-}\lim_i u_i 
\]
where each $u_i=\langle\varpi\times\varpi(\cdot,e)\xi'_i,\xi'_i\rangle$ is almost periodic.
Thus if $s\in G\setminus N_u$, $\langle\varpi\times\varpi(s,e)\xi'_i,\xi'_i\rangle\not=
1$ for some $i$, so $s\in G\setminus N_\cmap$, i.e.\ $N_\cmap\subseteq N_u$.
\end{proof}



\begin{example}\label{ex:amkT}
{\bf (i)} The following is motivated by \cite[Prop.\ 2.6.5]{ceccherinisc}.  Let
$F$ be a finite simple group, and consider the wreath product 
$G=F^{\oplus\Zee}\rtimes\Zee$.  This is readily seen to be amenable and finitely
generated.  The only subgroups of $N=F^{\oplus\Zee}\rtimes\{0\}$ that are normal in $G$
are $\{e\}$ and $N$ itself.  Hence $N$ is in the kernel of any homomorphism into
a finite group.  By way of Corollary \ref{cor:tdmap} (ii), we see that $N=N_\cmap$.
However this is an amenable discrete group so $N_\amTr=\{e\}$.

{\bf (ii)} Let $\kay$ be a local field and $G$ be one of $\SL_n(\kay)$ for $n\geq 3$, or
$\Sp_{2n}(\kay)$ for $k\geq 2$.  
Then $G$ has property (T); see \cite[\S\S1.4-1.5]{bekkadv}.  
The only proper normal subgroup of $G$ is the centre and 
$G$ contains a subgroup $S$ which is isomorphic to 
$\SL_2(\kay)$.  If $\kay$ has chracteristic $0$, then the proof of \cite{vonneumannw}
shows that $S\subseteq N_\cmap$, hence $N_\cmap=\alg(\kay)$.  
As noted in Remark \ref{rem:amkmap} (ii), $\cstar(G)$ is nuclear.
Hence $N_\Tr=N_\amTr=N_\cmap=\alg(\kay)$, and $\tpos(G)=\{1\}$.
\end{example}

\subsection{Property (F)}
The following is of independent interest, and is for use in the next theorem.

\begin{lemma}\label{lem:compression}
Let $u\in\tpos(G)$ and $K$ be a compact normal subgroup of $G$ for which
$\int_K u(k)\,dk\not=0$.  Then
\[
u_K=\frac{1}{\int_K u(k)\,dk}\int_Ku(\cdot k)\,dk\in\tpos(G)\text{ with }N_{u_K}=KN_u.
\]
If, further, $u\in\tpos_{\mathrm{am}}(G)$, then $u_K\in \tpos_{\mathrm{am}}(G)$ too.
\end{lemma}

\begin{proof}
We let $m_K$ denote the Haar probability measure on $K$, regarded as a measure on 
$G$.  Then $P_K=\pi_u(m_K)=\int_K\pi(k)\,dk$  is a projection on $\fH_u$ in the centre
of $\pi_u(G)''$.  Then $u_K$ is the trace $u$ compressed by $P_K$ and
normalized, provided the compression is non-zero.  Then Lemma \ref{lem:kernel}
shows that $N_{u_K}=\ker\pi_{u_K}$.  Here $\pi_{u_K}=P_K\pi_u(\cdot)|_{P_K\fH_u}$,
so $\ker\pi_{u_K}\supseteq KN_u$.  
Now if $s\in G\setminus KN_u$, then $\pi_u(s)\not\in \pi_u(K)$, so $\pi_{u_K}(s)\not=P_K$,
whence $s\not\in \ker\pi_{u_K}$.

We then have that $\til{u}_K$ is the normalized compression by
$P_K\otimes P_K$.  Hence if we have weak containment $\pi_{\til{u}}\prec
\varpi\otimes\varpi$, then we have weak containments
\[
\pi_{\til{u}_K}\prec Z_K\varpi(\cdot)|_{Z_K\fH_\varpi}\times Z_K\varpi(\cdot)|_{Z_K\fH_\varpi}
\prec\varpi\otimes\varpi
\]
where $Z_K=\varpi(m_K)$. In other words, $u_K$ is amenable provided that $u$ is.
\end{proof}

We say that $G$ has the {\it factorization property} , or {\em property (F)},
if the left-right regular
representation $\lam\cdot\rho:G\times G\to\Un(\bl^2(G))$, given by
\[
\lam\cdot\rho(s,t)h(x)=\lam(s)\rho(t)h(x)=h(s^{-1}xt)\Delta(t)^{1/2}
\]
satisfies that $\langle \lam\cdot\rho(\cdot,\cdot)h|h\rangle\in\pos_{\min}(G\times G)$
for each $h$;  i.e.\ we have weak containment $\lam\cdot\rho\prec \varpi\times\varpi$
so $\lam\cdot\rho$
induces a representation of $\cstar(G)\otimes_{\min}\cstar(G)$. In the case when $G$ is amenable, this is equivalent to the point mass $\delta_e$ being an amenable trace on $\cstar(G)$.

\begin{theorem}\label{theo:facpro}
\begin{itemize}
\item[(i)] If $G\in[\csin]$ and has property (F), then $G$ is amenably
tracially separated.
\item[(ii)] If $G$ is totally disconnected  and amenably tracially 
separated, then $G$ has property (F).
\end{itemize}
\end{theorem}

\begin{proof}
{\bf (i)}   If $U$ is a relatively compact
conjugation-invariant symmetric neighbourhood of the identity then
as in the proof of Proposition \ref{prop:trasin} (i), we see that
\[
\til{u}_U(s,t)=\frac{m(st^{-1}U\cap U)}{m(U)}
=\frac{m(sUt^{-1}\cap U)}{m(U)}
=\frac{\langle\lam(s)\rho(t)1_U,1_U\rangle}{m(U)}.
\]
Property (F) implies that $u_U\in\tpos_{\mathrm{am}}(G)$.
Furthermore the proof of  Proposition \ref{prop:trasin} (i) also
shows that this set of such traces establishes amenable tracial separation.

{\bf (ii)} We let $H$ be a compactly generated open subgroup.
Being tracially separated we have $H\in[\csin]$ by Corollary \ref{cor:comgen}.

Fix a compact open normal subgroup $K$ in $H$.  If $s\in H\setminus K$, then
$sK\subset H\setminus\{e\}=
\bigcup_{u\in\tpos_{\mathrm{am}}(H)}(H\setminus N_u)$, there
are $u_1,\dots,u_n$ in $\tpos_{\mathrm{am}}(H)$ so
$sK\subset H\setminus N_u$ where $u=\frac{1}{n}(u_1+\dots+u_n)$.
In particular, $s\not\in KN_u$.  Lemma \ref{lem:compression}
then shows that $s\not\in N_{u_K}$.

Let $L$ be a compact subset of $H\setminus K$.  For each
$s$ in $L$, find as above $u_s$ in $\tpos_{\mathrm{am}}(G)$
which is constant on cosets of $K$, $u_s(e)=1$ and $u_s(s)\not=1$.
Find $s_1,\dots,s_n$ in $L$ so $L\subseteq\bigcup_{j=1}^n s_jK$ and
let $u_L=\frac{1}{n+1}(1+u_{s_1}+\dots+u_{s_n})$ which is in
$\tpos_{\mathrm{am}}(H)$ with $u_L|_K=1$ and $|u_L(s)|<1$
for any $s$ in $L$.  
It follows from (\ref{eq:krocon}) that is closed under pointwise
multiplication. Thus $\tpos_{\mathrm{am}}(G)$ is closed under pointwise
multiplication.
Hence we can form a net $(u_L^k)$ indexed over the
directed product of $L$ in the increasing set of  compact subsets of $H\setminus K$
and $k$ in $\En$.  This net converges, uniformly on compact sets to $1_K$,
showing that $1_K\in\tpos_{\mathrm{am}}(H)$, hence $\til{1}_K\in \pos_{\min}(H\times H)$.
Recalling that $G$ is unimodular as noted in the proof of Proposition \ref{prop:asin},
as in the proof of (i), above, we see that
\[
\til{1}_K=\frac{1}{m(K)}\langle\lam\cdot\rho(\cdot,\cdot)1_K,1_K
\rangle\in\pos_{\min}(H\times H)\subseteq\pos_{\min}(G\times G).
\]
The latter containment holds as
$\tpos(H)\cong 1_H\tpos(G)\subseteq\tpos(G)$, and we use (\ref{eq:krocon}), above.

It follows from Lemma \ref{lem:tdsin} that translates of $1_K$, for all compact open normal
subgroups $K$ uniformly densely span $\fC_c(H)$, hence norm densely
$\bl^2(H)$.  But the union of $\bl^2(H)$ where $H$ ranges over the compactly generated
open subgroups of $G$ is dense in $\bl^2(G)$.    Hence for each $h$ in $\bl^2(G)$,
$\langle\lam\cdot\rho(\cdot,\cdot)h,h\rangle$ can be uniformly approximated on 
compact sets by sums of translations of elements $\til{1}_K$, above.  Thus $G$ admits property (F).
\end{proof}

\begin{remark} Wiersma has shown in \cite[Cor.\ 4.4]{wiersma}, that  
$\cmap$-groups, enjoy property (F).  
In light of Remark \ref{rem:amkmap} (i), part (ii) generalizes this fact for 
totally disconnected groups, and with a simpler proof.
\end{remark}

Part (i) of the next observation captures for discrete groups the titular property of \cite{kirchberg}, 
long since recovered by many other means.  

\begin{corollary}
\begin{itemize}
\item[(i)] If $G\in[\csin]$ with both properties (T) and (F), then $G\in[\cmap]$.
\item[(ii)] If $G$ is pro-discrete, then property (F) is equivalent to being amenably tracially separated.
\end{itemize}
\end{corollary}

\begin{proof}
To see (i) we apply Theorems \ref{theo:facpro} (i) and \ref{theo:amkT}.  To see (ii) we
appeal to Lemma \ref{lem:tdsin}, then Theorem \ref{theo:facpro} (ii).
\end{proof}

\begin{example} 
{\bf (i)} We display an example of a non-amenable, non-discrete, group which is 
amenably tracially separated, but not maximally almost periodic.

The group $\SL_2(\Que)\rtimes\langle\alp\rangle$ of Example \ref{ex:gammaF} (i)
may be written $\bigcup_{n=1}^\infty S_n\rtimes\langle\alp\rangle$  where each
$S_n=\SL_2(\Zee[\frac{1}{n!}])$, i.e.\ as an increasing union of residually finite groups.
That an increasing union of discrete property (F) groups admits property (F)
is noted in \cite[Prop.\ 7.3]{ozawa}.  
The group $G=\Gamma\rtimes F$ of
Example \ref{ex:gammaF} (i), admits a separating family of quotients, each discrete
with property (F).  Hence Theorem \ref{theo:facpro} (i) and (\ref{eq:krocon}) show that 
$G$ is amenably
tracially separated, while Theorem \ref{theo:facpro} (ii) (or \cite[Cor.\ 5.5]{wiersma}) shows that 
$G$ has property (F).

Notice  $G=\Gamma\rtimes F$ we have $N_\amTr=\{e\}\subsetneq 
\Gamma\rtimes\{e\}=N_\cmap$.

{\bf (ii)} The class of property (T) $\csin$-groups contains groups which are not products of
discrete and compact.

Let $\SL_3(\Zee)$ act on $\Tee^3\cong\what{\Zee^3}$ in the obvious manner, and form
$G=(\Tee^3)^\En\rtimes \SL_3(\Zee)$ by letting $\SL_3(\Zee)$ act diagonally on $(\Tee^3)^\En$, i.e.\
$s\cdot(z_n)_{n\in\En}=(s\cdot z_n)_{n\in\En}$.  This is easily checked to be a 
$\csin$-group with property (T) (an $(\varepsilon,K)$-vector for a representation
 averages over $(\Tee^3)^\En\rtimes \{I\}$ to be such over $\SL_3(\Zee)$).  
 
 Normal subgroups 
$((C_n^3)^n\times (\Tee^3)^{\{n+1,n+2,\dots\}})\rtimes \SL_3(n\Zee)$, where
$C_n$ is the $n$th roots of unity, manually show that $G\in[\crk]\subset[\cmap]$.
\end{example}

\subsection{Amenable reduced traces}
Since $\bl^2(G)\otimes\bl^2(G)\cong\bl^2(G\times G)$, we have that 
$\lam\times\lam=\lam_{G\times G}$ and 
\[
\cstar_r(G\times G)=\cstar_{\lam\times\lam}=\cstar(G)\otimes_{\min}\cstar(G).
\]
We consider the set of
{\em reduced amenable traces}
\[
\tpos_{r,\mathrm{am}}(G)=\{u\in\tpos_r(G):\til{u}\in\pos_r(G\times G)\}.
\]
For discrete $G$, Lance \cite[Theo.\ 4.2]{lance} showed that
$\cstar_r(G)$ is nuclear if and only if $G$ is amenable.
More recently, Ng \cite{ng} proved for general locally compact $G$ that
$G$ is amenable if and only if $\cstar_r(G)$ is nuclear and admits a trace.
The following captures the latter result, but the use of amenable traces
allows a simpler proof.

\begin{theorem}\label{theo:amtam}
The following are equivalent:
\begin{itemize}
\item[(i)] $G$ is amenable;
\item[(ii)] $\tpos_{r,\mathrm{am}}(G)\not=\varnothing$; and
\item[(iii)]  $\cstar_r(G)$ is nuclear and $\tpos_r(G)\not=\varnothing$.
\end{itemize}
\end{theorem}

\begin{proof}
Suppose (i) holds.  Then  $1\in\tpos(G)=\tpos_r(G)$ satisfies
$\til{1}=1\times 1$ which gives (ii).  Also $\cstar(G)=\cstar_r(G)$
is nuclear (see, for example, \cite{paterson}), so (iii) holds.

Suppose (ii) holds.  Let
$G_D=\{(s,s):s\in G\}\cong G$.  If $u\in  \tpos_{r,\mathrm{am}}(G)$,
then Lemma \ref{lem:redpd} provides that $1=\til{u}|_{G_D}\in\pos_r(G)$, so $G$ 
is amenable, i.e.\ we have (i).

If (iii) holds, then $\tpos_{r,\mathrm{am}}(G)=\tpos_r(G)$, so we get (ii).

[We can also use Theorem \ref{theo:kennedyr} (\cite{kennedyr}) to prove that (iii) implies
(i).  Indeed $\tpos_r(G)\not=\varnothing$ implies that $R=R(G)$ is open.  If $\cstar_r(G)$ is nuclear,
then soo to is its quotient $\cstar_r(G/R)$.  We then appeal to \cite[Theo.\ 4.2]{lance}
to see that $G/R$ must also be amenable, hence $G=R$.]
\end{proof}

Notice that if $G$ is any non-amenable residually finite discrete group, then
\cite[Prop.\ 7.3]{ozawa} (see also Theorem \ref{theo:facpro}) provides that
$1_{\{e\}}\in\tpos_r\cap\tpos_{\mathrm{am}}(G)$, while 
$\tpos_{r,\mathrm{am}}(G)=\varnothing$.  Hence a reduced and amenable trace
need not be reduced amenable trace.

\section{Embeddability of group C*-algebras into simple AF algebras}

In this section, we apply our results on traces to study structural properties of group C*-algebras. We will primarily be interested in AF embeddability of the reduced group C*-algebra. Determining whether a C*-algebra is AF embeddable can be a very difficult problem since no abstract characterization of C*-subalgebras of AF algebras is known. As such, we will satisfy ourselves by studying embeddability in unital, simple AF algebras, where a deep theorem of Schafhauser can be applied \cite[Theorem A]{schafhauser}.

Since AF algebras are quasidiagonal, the problem of determining when the reduced group C*-algebra embeds into an AF algebra is deeply related to when the reduced group C*-algebra is quasidiagonal. In 2017, Tikuisis, White and Winter proved a discrete group $G$ is amenable if and only if $\cstar_r(G)$ is quasidiagonal \cite[Corollary C]{tikuisisww}, solving the Rosenberg conjecture in the affirmative. Moreover, Tikuisis, White and Winter show that if $G$ is countable and amenable, then $\cstar_r(G)$ embeds into an AF-algebra \cite[Corollary 6.6]{tikuisisww}. Schafhauser further strengthened this result by showing that $\cstar_r(G)$ embeds into the universal UHF algebra $\mathcal Q=\otimes_{n=1}^\infty \mathbb M_n$, which is a unital, simple AF algebra, for all countable, discrete groups $G$ \cite[Theorem B]{schafhauser}. Building upon Schafhauser's result, we characterize all second countable locally compact groups whose reduced group C*-algebra embeds into a unital, simple AF algebra.

We will require the following lemma, whose proof is similar to that of \cite[Theorem 1.3]{bekkals}.

\begin{lemma}\label{lem:weakcontain}
	Let $G$ be a locally compact group. For each $u\in \tpos(G)$, let $\pi_u$ denote the GNS representation of $u$.
	\begin{enumerate}
		\item[(i)] $S_{\Tr}=\{\pi_u : u\in \tpos(G)\}$ weakly contains $\lambda_{G/N_{\Tr}}$;
		\item[(ii)] $S_{\Tr}^r=\{\pi_u : u\in \tpos_r(G)\}$ 
		weakly contains $\lambda_{G/N_{\Tr}^r}$;
		\item[(iii)] $S_{\mathrm{amTr}}=\{\pi_u : u\in \tpos_{\mathrm{am}}(G)\}$ 
		weakly contains $\lambda_{G/N_{\amTr}}$.
	\end{enumerate}
\end{lemma}

\begin{proof}
	The proofs of these three claims are nearly identical. We shall only prove (i).
		
	Let $q\colon G\to G/N_{\Tr}$ denote the canonical quotient map. For each $u\in \tpos(G)$, we let $u'$ be the unique element of $P(G/N_{\Tr})$ such that $u=u'\circ q$ and $\pi_u'$ denote the unique representation of $G/N_{\Tr}$ such that $\pi_u=\pi_{u}'\circ q$. Further, we set $\tpos'(G)=\{u' : u\in \tpos(G)\}$.  Since $\tpos'(G)$ is closed under pointwise multiplication, \cite[Lemma 1.1]{bekkals} implies there exists a net $\{v_\alpha'\}$ contained in the positive cone generated by $\tpos'(G)$, $\pos_{S_1}(G/N_{\Tr})$, so that $v_\alp'\,d m_{G/N_{\Tr}}$ converges to the point mass $\delta_{q(e)}$ with respect to the topology $\sigma({\rm M}(G/N_{\Tr}), \cont_c(G/N_{\Tr}))$.  Let $f\in \cont_c(G/N_{\Tr})$.  If we consider the natural action of $\cstar(G/N_{\Tr})$ on ${\rm B}(G/N_{\Tr})$, we then have $f\cdot v_\alp'\cdot f^*\in \pos_{S_1}(G/N_{\Tr})$ for every $\alpha$. Since
\begin{align*}	
\lim_\alp \|f\cdot v_\alp'\cdot f^*\|_{\fsal(G/N_{\Tr})} &=\lim_\alp f\cdot v_\alp'\cdot f^*(q(e))
=\lim_\alp\langle f\cdot v_\alp'\cdot f^*,\delta_{q(e)}\rangle \\
&=\lim_\alp\langle v_\alp',f^*\ast f\rangle=f^*\ast f(q(e))
\end{align*}	
	we may assume that the net $(f\cdot v_\alp'\cdot f^*)$ is bounded in ${\rm B}(G/N_{\Tr})$. We also calculate that for all $g\in \cont_c(G/N_{\Tr})$ that
\begin{align*}
\lim_\alp \langle f\cdot v_\alp'\cdot f^*,g\rangle&=\lim_\alp \langle v_\alp', f^**g*f\rangle \\
&=f^**g*f(e)=\langle \lambda_{G/N_{\Tr}}(g)f,f\rangle.
\end{align*}		
	It follows that $(f\cdot v_\alp'\cdot f^*)$ converges to the element $u'=\langle\lambda_{G/N_{\Tr}}(\cdot)f,f\rangle$ in the weak* topology on ${\rm B}(G/N_{\Tr})$. Hence, $\lambda_{G/N_{\Tr}}$ is weakly contained in $S_1'=\{\pi_u' : u\in \tpos(G)\}$, viewed as representations of $G/N_{\Tr}$, and $u'=\lim_\alp f\cdot v_\alp'\cdot f^*$ uniformly on compact
	subsets of $G/N_{\Tr}$.  But then $u=u'\circ q$ is the uniform limit on compact subsets of $G$
	of $((f\cdot v_\alp'\cdot f^*)\circ q)$, and this implies (i).
\end{proof}

The following theorem greatly generalizes \cite[Cor.\ 2.10]{beltitab}, which asserts the C*-algebra of a connected, second countable, locally compact, solvable group embeds into a simple, unital AF algebra if and only if the group is abelian.

\begin{theorem}\label{theo:redAF}
	Let $G$ be a second countable locally compact group. The following are equivalent.
	\begin{enumerate}
		\item[(i)] $\cstar_r(G)$ embeds into a simple, unital AF algebra.
		\item[(ii)] $\cstar_r(G)$ admits a faithful amenable trace.
		\item[(iii)] $G$ is amenable and tracially separated.
	\end{enumerate}
\end{theorem}

\begin{proof}

(i) $\Rightarrow$ (ii).  A unital AF algebra is always nuclear and has a faithful trace.  Hence
any C*-subalgebra admits an amenable trace.

(ii) $\Rightarrow$ (iii). Suppose $\cstar_r(G)$ admits a faithful amenable trace $u$.  It is immediate
from Theorem \ref{theo:amtam} that $G$ is amenable.  Since $u$ is faithful, $N_u=\{e\}$ and
$G$ is tracially separated.

(iii) $\Rightarrow$ (i): Suppose $G$ is amenable and tracially separated. Since $G$ is second countable, the group algebra $\bl^1(G)$ is separable. Let $(f_n)_{n=1}^\infty$ be a sequence in $\bl^1(G)$ with dense range. Since $\lambda$ is weakly contained in $\{\pi_u : u\in \tpos(G)\}$ by the previous lemma, for each natural number $n$, we can find $u_n\in \tpos(G)$ such that $\|\pi_{u_n}(f)\|\geq \frac{1}{2}\|\lambda(f_n)\|$. Then $\{\pi_{u_n} : n\in \mathbb N\}$ separates points of $\cstar_r(G)$ and, so, $u:=\sum_{n=1}^\infty \frac{1}{2^n} u_n$ is a faithful trace on the separable, nuclear C*-algebra $\cstar_r(G)$. As the reduced group C*-algebra of an amenable, second countable locally compact group satisfies the UCT by a result of Tu \cite{tu}, we conclude that $\cstar_r(G)$ embeds into a AF algebra by Schafhauser's theorem \cite[Theorem A]{schafhauser}.
\end{proof}

If we further restrict our attention to only second countable, compactly generated groups $G$, the results obtained in Section \ref{sec:traker} combined with Theorem \ref{theo:redAF} imply $\cstar_r(G)$ embeds into a simple, unital AF algebra if and only if $G$ is amenable and SIN.

As an immediate consequence of Theorem \ref{theo:redAF}, one obtains that if $G$ is an amenable, tracially separated, second countable locally compact group, then $\cstar_r(G)$ is quasidiagonal.  We can remove the assumption of second countability.  We begin by
mildly expanding on an observation of Lau and Losert \cite[Rem.\ 14(b)]{laul}
which gives a variant of the Kakutani-Kodaira Theorem.

\begin{lemma}\label{lem:laul}
Let $H$ be a $\sig$-compact locally compact group. Then there is decreasing net
$(K_i)$ of compact normal subgroups for which $\bigcap_iK_i=\{e\}$ and each $H/K_i$
is metrizable, hence second countable.
\end{lemma}

\begin{proof}
For any $f$ in $\cont_c(H)$, $\sig$-compactness of $H$ shows that the set of translates
$\{f(r\cdot t)\}_{r,t\in G}$ is separable in $\cont_0(H)$. 
For any $s\in H\setminus\{e\}$, let $f_s\in\cont_c(G)$ be so $f_s(s)=0$ while $f_s(e)=1$.
Then for each finite subset $F\subset H$, the smallest closed 
subalgebra containing all translates of $\{f_s\}_{s\in F}$ in $\cont_0(H)$ is separable,
and, by \cite[Lem.\ 12]{laul}, isomorphic to $\cont_0(G/K_F)$ for a compact normal
subgroup $K_F$, so $H/K_F$ is $\sig$-compact.  
Bounded sets in the dual space are hence metrizable, so $G/K_F$ is
metrizable.  With our choices, $K_F\supseteq K_{F'}$ if $F\subseteq F'$.  In a metrizable
space any compact set is second countable, so $\sig$-compact metrizable spaces
are second countable.
\end{proof}

Comments in \cite[Ex. 1.19]{leidermanmt} show that we cannot relax the assumption of
$\sig$-compactness, above.

\begin{corollary}\label{cor:redQuasi}
	If $G$ is an amenable, tracially separated locally compact group, then $\cstar_r(G)$ is quasidiagonal.
\end{corollary}

\begin{proof}
Let $H$ be a compactly generated, hence $\sig$-compact subgroup of $G$.  Let
$(K_i)$ be as in Lemma \ref{lem:laul}.  Then each $P_i=\lam(m_{K_i})$ is a central
projection in the multiplier algebra of $\cstar_r(H)$ with $\cstar_r(H)P_i\cong\cstar_r(H/K_i)$, 
which is quasi-diagonal as a consequence of Theorem \ref{theo:redAF} (i).  Furthermore
$\bigcup_i \cstar_r(H)P_i$ is dense in $\cstar_r(H)$, showing that the latter is quasi-diagonal,
thanks to, for example, \cite[Prop.\ 7.1.9]{brownozawa}.
But then the union $\bigcup_H \cstar_r(H)$, over all compactly generated open subgroups, is 
norm dense in $\cstar_r(G)$, giving the desired result.
\end{proof}

Interestingly, the condition of $G$ being tracially separated in Corollary \ref{cor:redQuasi}  cannot be dropped. Indeed, examples of amenable Lie groups whose reduced C*-algebra is not quasidiagonal are given in \cite{beltitaa}. 

As a further application of the techniques we have developed in this paper, we show that full C*-algebras of property (T) groups do not embed inside of simple, unital AF algebras.

\begin{proposition}
	Let $G$ be a non-compact locally compact group with property (T). Then $\cstar(G)$ does not embed inside of a simple, unital AF algebra.
\end{proposition}

\begin{proof}
	Suppose towards a contradiction that $\cstar(G)$ embeds into a simple, unital AF algebra. Then $\cstar(G)$ admits a faithful, amenable trace. Consequently, $N_{\amTr}=\{e\}$ implying that $G$ is maximally almost periodic by Theorem \ref{theo:amkT}. So $G$ has the factorization property by \cite[Cor.\ 4.4]{wiersma}. Since groups with property (T) are compactly generated and $G$ is tracially separated, we also have that $G$ is a SIN group. It now follows from \cite[Prop.\ 3.1]{wiersma} that $\cstar(G)$ is not exact since SIN groups are inner amenable, contradicting the fact that C*-subalgebras of AF algebras are exact.
\end{proof}

\section{Questions}

Proposition \ref{prop:trasin} shows that $N_\Tr\subseteq N_\csin$, and these coincide
if $G$ is compactly generated.

\begin{question}
Is there a locally compact group for which $N_\Tr\subsetneq N_\csin$?
\end{question}



Theorem \ref{theo:facpro} shows that an amenably tracially separated totally 
disconnected group admits property (F).  Furthermore, if $G$ is
(amenably) tracially separated, then $G_0=V\times K$ is amenable.

\begin{question}
If $G$ is amenably tracially separated, does it admit property (F)?
\end{question}


\vfill



\end{document}